\documentclass[a4paper,10pt]{article}
\usepackage[english]{babel}
\usepackage{amssymb}
\usepackage{amsmath}
\usepackage{indentfirst}
\usepackage[dvips]{graphicx}
\usepackage{epsfig}
\usepackage{lscape}
\usepackage{hyperref}
\usepackage{color}
\usepackage{amsthm}

\newtheorem{cor}{Corollary}[section]

\newtheorem{lemma}[cor]{Lemma}
\newtheorem{teo}[cor]{Theorem}
\newtheorem{algo}{Algorithm}[section]
\newtheorem{rem}[cor]{Remark}
\newtheorem{prob}[cor]{Problem}
\newtheorem{ass}[cor]{Assumption}

\newcommand{\R}{\mathbb{R}}
\newcommand{\mbR}{\mathbb{R}}

\def\mbR{\mathbb{R}}
\def\mbN{\mathbb{N}}
\def\eps{\varepsilon}

\def\mmm{\mathcal}

\def\mathref#1{\ifmmode\mathrm{(\ref{#1})}\else{\rm(\ref{#1})}\fi} 
\def\nref#1{\ifmmode\mathrm{\ref{#1}}\else{\rm\ref{#1}}\fi} 
 
\def\md{\pd^{\bullet}_{t}}
\def\pd{\partial}

\def\eoc{{\rm eoc}}
\def\intertime{\bar{t}}
\def\embconst{\hat{C}(S^1)}
\numberwithin{equation}{section}
\title{Curve shortening flow coupled to lateral diffusion}

\author{Paola Pozzi  \thanks{Fakult\"at f\"ur Mathematik, Universit\"at Duisburg-Essen,
Thea-Leymann-Stra\ss e 9,
45127 Essen, Germany, \url{paola.pozzi@uni-due.de}} \and and Bj\"orn Stinner  \thanks{Mathematics Institute,
Zeeman Building, 
University of Warwick,
Coventry CV4 7AL, 
United Kingdom, \url{Bjorn.Stinner@warwick.ac.uk}}
}

\begin{document}

\maketitle

\begin{abstract}
We present and analyze a semi-discrete finite element scheme for a system consisting of a geometric evolution equation for a curve and a parabolic equation on the evolving curve. More precisely, curve shortening flow with a forcing term that depends on a field defined on the curve is coupled with a diffusion equation for that field. The scheme is based on ideas of \cite{D99} for the curve shortening flow and \cite{DE07} for the parabolic equation on the moving curve. Additional estimates are required in order to show convergence, most notably with respect to the length element: While in \cite{D99} an estimate of its error was sufficient we here also need to estimate the time derivative of the error which arises from the diffusion equation. Numerical simulation results support the theoretical findings.
\end{abstract}

\bigskip
\noindent \textbf{Keywords:} curve shortening flow with a forcing term, surface PDE, finite element approximation, convergence 
\bigskip
 
\noindent \textbf{MSC(2010):} 65M15, 65M60, 35K65   

%%%%%%%%%%%%%%%%%%%%%
\section{Introduction}

We aim for approximating the following problem: Given a closed initial curve $\Gamma_{0}$ and a function $c_{0} : \Gamma_{0} \to \mbR$ find a moving closed curve $\{ \Gamma(t) \}_{t \in [0,T]} \subset \mbR^2$ and a family of fields $c(t) : \Gamma(t) \to \mbR$, $t \in [0,T]$, such that
\begin{align}
v &= \kappa + f(c), \label{mod1:evol} \\
\md c - c \kappa v &= c_{ss}, \label{mod1:spde} \\
\Gamma(0) = \Gamma_{0}, & \quad c(0) = c_{0}. \label{mod1:IC}
\end{align}
Here, $s$ is an arc-length parameter of the actual curve $\Gamma(t)$, $v$ is the (scalar) velocity in the direction of a unit normal field $\nu$, $\kappa$ is the (scalar) curvature, $f : \mbR \to \mbR$ is a coupling function, and $\md$ is the material derivative ($\md c = \partial_t c + v \partial_\nu c$ if $c$ is smoothly extended away from $\Gamma$).

The system consisting \eqref{mod1:evol}, \eqref{mod1:spde} can fairly be regarded as the simplest system coupling a geometric evolution equation to an equation for a conserved field on the evolving manifold. We don't have any specific application in mind for \eqref{mod1:evol}, \eqref{mod1:spde}. But more sophisticated geometric evolution equations and parabolic PDEs on the moving manifold feature, for instance, in cell biology as an effective approach to cell motility \cite{EllStiVen2012}. Problems in soft matter physics such as the relaxation dynamics of two-phase biomembranes can also be modeled by such type of systems \cite{EllSti2010,EllSti2013}. From a mathematical point of view, the evolution of pattern forming PDE systems on deforming surfaces is of general interest, for instance, see \cite{PlaSan-GarPad2004}.

Working in a parametric setting we assume that the curves can be parametrized by a family of functions $u(t) : S^1 \to \mbR^2$, i.e., $\Gamma(t) = u(S^1,t)$. For the initial curve we write $\Gamma_0 = u_0(S^1)$. By $(\cdot)^\perp$ we denote the counter-clockwise rotation by 90 degree in $\mbR^2$. We write $\tau = u_{x} / |u_{x}|$ for a unit tangent field and assume that the orientation is such that $\nu = u_{x}^\perp / |u_{x}|$. For convenience, the field $c$ on the evolving curve will be denoted by $c$ again after transformation to the parameter space. The strong formulation of the geometric equation in the parameter setting then is
\begin{equation} \label{eq:spgeoA}
 0 = u_t - \frac{1}{|u_{x}|} \Big{(} \frac{u_{x}}{|u_{x}|} \Big{)}_{x} - f(c) \frac{u_{x}^\perp}{|u_{x}|} = u_t - \frac{1}{|u_{x}|} \tau_x - f(c) \nu
\end{equation}
while for the PDE on the evolving curve we obtain
\begin{equation} \label{eq:spspdeA}
 0 = c_t + c \frac{|u_x|_t}{|u_x|} - \frac{1}{|u_x|} \Big{(} \frac{c_x}{|u_x|} \Big{)}_x.
\end{equation}

In order to approximate the solution let $Y_h$ denote a finite element space (details will be provided later on in Section \ref{sec:scheme}) and let $X_h = Y_h^2$. Then consider the problem of finding functions $u_h(\cdot,t) \in X_h$ and $c_h(\cdot,t) \in Y_h$, $t\in [0,T]$, such that $u_h(\cdot,0) = u_{h0} := I_h(u_{0})$, $c_h(\cdot,0) = c_{h0} := I_h(c_0)$, and such that for all $\varphi_{h} \in X_h$ and $\zeta_{h} \in Y_h$ at almost all times $t \in [0,T]$
\begin{align}
\int_{S^1} I_h\big{(} u_{ht} \cdot \varphi_{h} \big{)} |u_{hx}| \, dx + \int_{S^1} \frac{u_{hx}}{|u_{hx}|} \cdot \varphi_{hx} \, dx \, &= \int_{S^1} I_h \big{(} f(c_{h})  \varphi_{h} \big{)}\cdot u_{hx}^\perp \, dx, \label{eq:wpgeoh} \\
\frac{d}{dt} \Big{(} \int_{S^1} c_{h} \zeta_{h} |u_{hx}| \, dx \Big{)} + \int_{S^1} \frac{c_{hx} \zeta_{hx}}{|u_{hx}|} \, dx &= 0. \label{eq:wpspdeh}
\end{align}
Here, $I_h$ stands for the interpolation operator for both scalar and vector valued functions. 

With regards to the equation \eqref{eq:spspdeA} for $c$, the approximation by \eqref{eq:wpspdeh} is inspired by \cite{DE07}. The resulting scheme is intrinsic in the sense that it doesn't require any knowledge about the parametrization but only the positions of the vertices that are given in terms of $u_h$ (see Algorithm \ref{algo:solver} below). However, for the numerical analysis we cannot resort to the methods in \cite{DE07} because the moving curve $\Gamma(t)$ is not explicitly given but by the solution $u$ of the geometric equation \eqref{eq:spgeoA}. Its approximation by \eqref{eq:wpgeoh} is based on \cite{Dzi1991} where a scheme for two-dimensional surfaces is presented. The one-dimensional semi-discrete case but with anisotropic surface energy has been analyzed in \cite{D99} (evolution in a plane) and in \cite{PaolaACSF} (higher co-dimension), see also \cite{Dzi1994} for the isotropic case. In addition, there is the forcing term $f(c) u_{x}^\perp / |u_{x}|$ which is of lower order but, because of the $c$ dependence, requires a coupling of the error estimates for $u$ to those for $c$. 

Regarding the estimate for $c$, the main difficulty arises from the term $c |u_x|_t / |u_x|$ in \eqref{eq:spspdeA}. The error of the length element $|u_x| - |u_{hx}|$ already had to be estimated in the $L^\infty([0,T],L^2(S^1)$ norm when proving convergence of the approximation to curve shortening flow in \cite{D99}. However, here we need an estimate for the time derivative of the length element $|u_x|_t - |u_{hx}|_t$. The key observation is that $|u_x|_t$ can be estimated in terms of the squared velocity and the length element, see \eqref{(3.5)} in Lemma \ref{lem:est_le} below. Mimicking these calculations for the error $|u_x|_t - |u_{hx}|_t$ is the content of the novel Lemma \ref{lem:err_est_A} which subsequently proves sufficient to obtain suitable estimates for $c - c_h$. Our results are summarized by:

\begin{teo}
\label{Thm5.3}
Under Assumption \ref{ass:sol} there exists $h_0>0$ such that for all $0 < h \leq h_0 $ there exists a unique solution $(u_h,c_h)$ of \eqref{eq:wpgeoh}, \eqref{eq:wpspdeh}, and the error between the smooth solution and the discrete solutions can be estimated as follows: 
\begin{align}
\int_0^T \int_{S^1} \Big{(} |u_t - u_{ht}|^2 + |c_x - c_{hx}|^2 \Big{)} dx dt  & \leq C \, h^2, \label{eq:theoest1}\\
\sup_{t \in [0,T]} \int_{S^1} \Big{(} |\tau - \tau_h|^2 + |c - c_h|^2 + \big{(} |u_x| - |u_{hx}| \big{)}^2 \Big{)} dx &\leq C \, h^2, \label{eq:theoest2}
\end{align}
with a constant $C>0$. 
The constant depends on the final time $T$, on the bounds $\| f \|_{L^{\infty}}$ and $\| f' \|_{L^{\infty}}$ of the coupling function, on the regularity and bounds $\| u \|_{W^{1,\infty}([0,T],H^2(S^1))}$, $\| c \|_{W^{1,\infty}([0,T],H^1(S^1))}$, and $\| c \|_{L^\infty([0,T],H^2(S^1))}$ of the solution (which includes the bounds $\| u_0 \|_{H^2(S^1)}$ and $\| c_0 \|_{H^1(S^1)}$ of the initial values), on the bound $C^{**}$ from below of the length element, see \eqref{eq:ass_lengthelement} in Assumption~\ref{ass:sol}, and on the constant $\bar{C}$ ruling the grid regularity (cf.~\eqref{(4.1)}). 
\end{teo}

Our proof follows the lines of \cite{D99} on anisotropic curve shortening flow though we should mention that for the isotropic curve shortening flow other ideas and techniques have also been used, for instance, see \cite{DecDzi1995}. From a practical point of view, mesh degeneration is an important problem for long-time simulations. We won't address this issue here but for ideas to move vertices in tangential direction as appropriate we refer to \cite{MikSevBal2010}, \cite{BarGarNue2011}, and \cite{BarGarNue2010}. In \cite{BalMik2011} an additional forcing term is accounted for, see also \cite{ChambNovaga} for analytical results on such a problem. Also with regards to PDEs on evolving surfaces there are other methods. For instance, in \cite{OlsReu2014A} a surface reconstruction is used which is based on a fixed bulk mesh and in \cite{LeuLowZha2011} a grid based particle method. Of course, there are also other approaches to surfaces PDEs and geometric PDEs which are not based on any parametrization but on level sets, phase field, or other ideas. We here only refer to the overviews \cite{DDE} and \cite{DziEll2013}. 

We start with specifying the assumptions on the solution to the continuous problem and showing some properties in Section \ref{sec:ass}. After, we carefully describe the finite element approach and, proceeding analogously to the continuous case, show some properties of the semi-discrete solution. Section \ref{sec:errestim} then contains the technical estimates required for convergence which is stated in the section after. In the final section we report on numerical simulation results which support the findings.

\bigskip
\textbf{Acknowledgements:} The authors would like to thank the Isaac Newton Institute for Mathematical Sciences, Cambridge, for support and hospitality during the programme \emph{Coupling Geometric PDEs with Physics for Cell Morphology, Motility and Pattern Formation} where work on this paper was undertaken.

%%%%%%%%%%%%%%%%%%%%%
\section{The continuous problem}
\label{sec:ass}

Here and in the following sections, constants which, in general, will vary from line to line in the various computations will be denoted by capital $C$. Moreover we occasionally use the abbreviation
\begin{equation} \label{eq:r}
 r = f(c) \frac{u_{x}^\perp}{|u_{x}|}.
\end{equation}

The finite element approximation consisting of \eqref{eq:wpgeoh} and \eqref{eq:wpspdeh} emerges from the following weak formulation of the system \eqref{eq:spgeoA} and \eqref{eq:spspdeA}:

\begin{prob}[Weak problem] \label{prob:weak}
Find functions $u : S^1 \times [0,T] \to \mbR^2$ and $c : S^1 \times [0,T] \to \mbR$ such that $u(\cdot,0) = u_0$, $c(\cdot,0) = c_0$, and such that for all test functions $\varphi : S^1 \to \mbR^2$ and $\zeta : S^1 \to \mbR$ and almost all times $t \in [0,T]$ 
\begin{align}
\int_{S^1} u_{t} \cdot \varphi |u_{x}| \, dx + \int_{S^1} \frac{u_{x}}{|u_{x}|} \cdot \varphi_{x} \, dx \, &= \int_{S^1} f(c) \varphi \cdot u_{x}^\perp \, dx, \label{eq:wpgeo} \\
\frac{d}{dt} \Big{(} \int_{S^1} c \zeta |u_{x}| \, dx \Big{)} + \int_{S^1} \frac{c_x \zeta_x}{|u_{x}|} \, dx &= 0. \label{eq:wpspde}
\end{align}
\end{prob}

Note that if $\zeta : S^1 \times [0,T] \to \mbR$ is a time dependent test function then \eqref{eq:wpspde} becomes
\begin{equation} \label{eq:wpspde_zt}
\frac{d}{dt} \Big{(} \int_{S^1} c(t) \zeta(t) |u_{x}(t)| \, dx \Big{)} + \int_{S^1} \frac{c_x \zeta_x}{|u_{x}|} \, dx = \int_{S^1} c \zeta_t |u_x| \, dx.
\end{equation}
Clearly, we can not expect the flow to be eternal, since the flow might exhibits singularities in finite time (like the curve shortening flow). We thus make the following assumptions regarding existence, uniqueness, and regularity of the weak solution:

\begin{ass} \label{ass:sol}
Both $f$ and its derivative $f'$ are bounded, 
\[
 \| f \|_{C^0(\mbR)} \leq C, \quad \| f' \|_{C^0(\mbR)} \leq C \qquad \mbox{with some } C > 0.
\]
There is a unique solution $(u,c)$ of \eqref{eq:wpgeo}, \eqref{eq:wpspde} on the time interval $[0,T]$ with initial values $u(\cdot,0)=u_0(\cdot) \in H^2(S^1)$, $c(\cdot,0)=c_0(\cdot) \in H^1(S^1)$ which satisfies
\begin{align*}
u \, & \in W^{1,\infty}([0,T],H^2(S^1)), \\
c \, & \in W^{1,\infty}([0,T],H^1(S^1)) \cap L^\infty([0,T],H^2(S^1)).
\end{align*}
Moreover, there is a constant $C^{**}>0$ such that 
\begin{equation} \label{eq:ass_lengthelement}
 |u_{x}| \geq C^{**} \mbox{ on } S^1 \times [0,T].
\end{equation}
\end{ass}

\begin{rem}
There is a huge literature on the curve shortening flow (and more generally on the mean curvature flow), see for instance \cite{CZ} and \cite{Mantegazza}. There are also results for curve shortening flow with a forcing term. For instance, in \cite{Cesaroni} it is shown that if $f$ is smooth and the initial curve $u_{0}$ is embedded then the maximal existence time of a smooth solution is bounded from below by a quantity that depends on the initial data and $\| f \|_{\infty}$. There don't seem to exist any results on short time well-posedness, regularity, and long-time behavior for our specific type of problem. However we count upon the standard methods for proving short-time well-posedness for parabolic systems to work thanks to the relatively nice elliptic second order structure of the spatial part of the differential operator. We leave these analytical questions for future studies and here focus on approximating the solution as it is postulated in the above Assumption \ref{ass:sol}.
\end{rem}

From now on $(u,c)$ will always denote the solution as specified above. Note that direct consequences of Assumption \ref{ass:sol} are that 
\begin{equation} \label{cbounds}
\|c\|_{C([0,T],L^\infty(S^1))} \leq C, \quad \|c\|_{C([0,T],H^{1}(S^1))} \leq C, \quad \|c\|_{L^2([0,T],H^1(S^1))} \leq C
\end{equation}
with a constant $C>0$.

Although the bounds derived in the next lemma are implied  by the regularity assumptions imposed on the continuous solution, the derived equations and methods of proof will be important to derive discrete analogues later on.
\begin{lemma} \label{lem:est_le}
\begin{enumerate}
\item For the length element we have that
\begin{align}
\label{(3.5)}
|u_x|_t &=-|u_t|^2 \,|u_x| + u_t \cdot r \,|u_x|.
\end{align}
\item Furthermore, 
\begin{align} \label{Cstar}
|u_x| \leq C^{*} 
\end{align}
with a constant $C^{*}>0$. 
\item For ${\intertime} \in [0,T]$ we have that
\begin{equation} \label{eq:est_u_t}
\int_0^{\intertime} |u_t|^2 dt \leq C \quad \mbox{and} \quad \int_0^{\intertime} |u_t - r|^2 dt \leq C \qquad \mbox{on } S^1\,.
\end{equation}
\end{enumerate}
\end{lemma}
\begin{proof}
We have that
\begin{align*}
|u_x|_t = \tau \cdot (u_{t})_x = (\tau \cdot u_t)_x - \tau_x \cdot u_t = -(u_t -r) \cdot u_t \,|u_x|
\end{align*}
by \eqref{eq:spgeoA}, \eqref{eq:r}, and the fact that $u_t$ is a normal vector. The second claim follows from the boundedness of $f$ and a Gronwall argument applied to
\begin{equation}\label{miserve}
 |u_x|_t \leq -|u_t|^2 |u_x| +| u_t | \, | r| \,|u_x| \leq -\frac{1}{2}|u_t|^2 |u_x| +\frac{1}{2} |r|^2 \, |u_x| \leq \frac{1}{2} |f(c)|^2 \, |u_x|.
\end{equation}
Finally observe that from \eqref{(3.5)} we know that
$ |u_t|^2 \leq -\frac{|u_x|_t}{|u_x|}  + |u_t| \, |r| \leq - \frac{|u_x|_t}{|u_x|} + \frac{1}{2} |r|^2 + \frac{1}{2} |u_{t}|^2\,,$
whence $$ \frac{1}{2}|u_t|^2 |u_{x}| + |u_x|_t   \leq \frac{1}{2} |r|^2 |u_{x}|.$$
Integration and \eqref{Cstar} gives the third claim. 
\end{proof}

%%%%%%%%%%%%%%%%%%%%%
\section{Spatial discretization}
\label{sec:scheme}

Let $S^1= \bigcup_{j=1}^N S_j$ be a decomposition of $S^1$ into segments given by the nodes $x_j$. We think of $S_j$ as the interval $[x_{j-1}, x_{j}] \subset [0,2\pi]$ for $j=1, \ldots, N$. Here and in the following, indices related to the grid have to be considered modulo $N$. For instance, we identify $x_0=x_N$. Let $h_j=|S_j|$ and $h=\max_{j=1,\ldots, N}h_j$ be the maximal diameter of a grid element. We assume that for some constant $\bar{C}>0$ we have 
\begin{align}
\label{(4.1)}
h_j \geq \bar{C}h, \qquad |h_{j+1}-h_j|\leq \bar{C}h^2\,.
\end{align}
Clearly the first inequality yields $\bar{C}h_{j+1} \leq h_j \leq \frac{h_{j+1}}{\bar{C}} $.
For a discretization of \eqref{eq:wpgeo} we introduce the discrete finite dimensional spaces
\begin{equation*}
Y_h := \{ v \in C^0(S^1, \R) \, : \, v|_{S_j} \in P_1 (S_j),\, j=1 \cdots, N\}, \quad X_h = Y_h^2
\end{equation*}
of continuous periodic piecewise affine functions on the grid. The scalar nodal basis functions of $Y_h$ are denoted by $\varphi_j$, $j=1, \ldots, N$, and defined by $\varphi_j(x_i)=\delta_{ij}$. 

For a continuous function $v \in C^0 (S^1, \R)$ let $I_h v \in Y_h$ be the linear interpolate uniquely defined by $I_h v(x_i)= v(x_i)$ for all $i=1, \ldots,N$. For convenience we also denote the interpolation onto $X_h$ by $I_h$. We shall use the standard interpolation estimates (both for scalar and vector valued functions):
\begin{align}
\label{(4.2)}
\| v-I_h v \|_{L^2(S^1)} &\leq C h^k \|v \|_{H^k(S^1)} &\text{ for $k=1,2$}\,,\\
\label{(4.2)bis}
\| (v-I_h v)_x \|_{L^2 (S^1)} & \leq C h \| v \|_{H^2 (S^1)}\,,\\
\label{(4.2)extra}
 \| (I_{h} v)_{x} \|_{L^{2} (S^{1})} & \leq C \| v_{x} \|_{ L^{2}(S^{1})}\, . 
\end{align} 
Recall also the inverse estimates for any  $w_{h} \in Y_{h}$ and $j=1, \ldots, N$:
\begin{align}
\| w_{hx} \|_{L^{2}(S_{j})} & \leq \frac{C}{h_{j}} \| w_{h} \|_{L^{2} (S_{j})} & \quad \overset{\eqref{(4.1)}}{\Longrightarrow} \quad \|w_{hx} \|_{L^{2}(S^1)} & \leq \frac{C}{h} \| w_{h} \|_{L^{2} (S^1)}, \label{IE-1} \\
\| w_{h} \|_{L^{\infty} (S_{j})} & \leq \frac{C}{\sqrt{h_{j}}} \| w_{h} \|_{L^{2}(S_{j}) } & \quad \overset{\eqref{(4.1)}}{\Longrightarrow} \quad \| w_{h} \|_{L^{\infty} (S^1)} & \leq \frac{C}{\sqrt{h}} \| w_{h} \|_{L^{2}(S^1)} \label{IE-2}.
\end{align}

\begin{prob}[Semi-discrete Scheme] \label{prob:semidis}
Find functions $u_h(\cdot,t) \in X_h$ and $c_h(\cdot,t) \in Y_h$, $t\in [0,T]$, of the form
\begin{equation*}
u_h(x,t)= \sum_{j=1}^{N} u_j(t) \varphi_j(x), \quad c_h(x,t) = \sum_{j=1}^{N} c_j(t) \varphi_j(x)
\end{equation*} 
with $u_j(t) \in \R^2$ and $c_j(t) \in \mbR$, such that $u_h(\cdot,0) = u_{h0} := I_h(u_0)$, $c_h(\cdot,0) = c_{h0} := I_h(c_0)$, and such that for all $\varphi_{h} \in X_h$ and $\zeta_{h} \in Y_h$ at almost all times $t \in [0,T]$ \eqref{eq:wpgeoh} and \eqref{eq:wpspdeh} are satisfied. 
\end{prob}

Note that we may want to use a time dependent test function in the equation for $c_h$ of the form $$\zeta_h(x,t) = \sum_{j=1}^N \zeta_j(t) \varphi_j(x).$$ In analogy to \eqref{eq:wpspde_zt} equation \eqref{eq:wpspdeh} then becomes
\begin{equation} \label{eq:wpspde_zth}
 \frac{d}{dt} \Big{(} \int_{S^1} c_{h}(t) \zeta_{h}(t) |u_{hx}(t)| \, dx \Big{)} + \int_{S^1} \frac{c_{hx} \zeta_{hx}}{|u_{hx}|} \, dx = \int_{S^1} c_h \zeta_{ht} |u_{hx}| \, dx.
\end{equation}

Recalling that indices referring to the grid always are understood modulo $N$, let 
\begin{align*}
q_j = |u_j - u_{j-1}|, \qquad \tau_j = \frac{u_j-u_{j-1}}{q_j}, \qquad \nu_{j} = \tau_{j}^{\perp}.
\end{align*}
If we insert $\varphi_j$, $j=1, \ldots,N$, separately for each component of $\varphi_h$ in \eqref{eq:wpgeoh} then we get the following $2 \times N$ ordinary differential equations:
\begin{equation} \label{eq:odegeo}
\frac{q_{j}+q_{j+1}}{2} \dot{u}_{j} + \tau_{j} - \tau_{j+1} = \frac{1}{2} f(c_j) (u_{j+1}-u_{j-1})^\perp,
\end{equation}
and the initial values are given by $u_j(0) = u_0(x_j)$, $j=1,\ldots,N$. With
\begin{equation} \label{nubar}
\bar{\nu}_{j} := \frac{(u_{j+1}- u_{j-1})^{\perp}}{q_{j}+q_{j+1}} = \frac{(q_{j}\tau_{j}^{\perp} + q_{j+1}\tau_{j+1}^{\perp})}{q_{j}+q_{j+1}} = \frac{1}{q_{j}+q_{j+1}} (q_{j}\nu_{j}+ q_{j+1}\nu_{j+1})
\end{equation}
and 
\begin{align}\label{eq:defr_j}
r_j := f(c_j) \bar{\nu}_j
\end{align}
we can rewrite the system \eqref{eq:odegeo} with the initial condition as
\begin{equation} \label{eq:odegeoB}
\begin{array}{rcl}
\dot{u}_{j} + \frac{2}{q_{j}+q_{j+1}} (\tau_{j} -\tau_{j+1}) & = & r_{j}, \\
u_{j}(0) & = & u_0(x_{j}),
\end{array}
\quad \mbox{for }j = 1, \ldots, N.
\end{equation}
Define the piecewise constant function
\[
 h_d : S^1 \to \mbR, \quad h_d(x) = h_{j} \mbox{ for } x\in S_j.
\]
A short calculation shows that another equivalent formulation to \eqref{eq:wpgeoh} is
\begin{multline} \label{eq:wpgeohB}
\int_{S^1} u_{ht}|u_{hx}|\varphi_h\, dx + \frac{1}{6}\int_{S^1} u_{hxt} |u_{hx}| h_d^2 \varphi_{hx} \, dx + \int_{S^1}\frac{u_{hx}}{|u_{hx}|} \varphi_{hx} \, dx \\
= \int_{S^{1}} I_{h}(f(c_{h})) \varphi_h \cdot u_{hx}^{\perp} \, dx + \frac{1}{6} \int_{S^{1}} (I_{h}(f(c_{h})) )_{x} \varphi_{hx} \cdot u_{hx}^{\perp} h_{d}^{2} \, dx.
\end{multline}

Next we aim at giving the discrete equivalents of the results in Lemma \ref{lem:est_le}. 

\begin{lemma}
Let ${\intertime} \in (0,T]$ and assume that $(u_h,c_h)$ is a solution of \eqref{eq:wpgeoh}, \eqref{eq:wpspdeh} for $t \in [0,{\intertime}]$ such that $q_j(t) >0$ for all $j=1, \ldots, N$ and all $t \in [0,{\intertime}]$. 
\begin{enumerate}
 \item For $j=1,\ldots,N$ we have that
\begin{align}
\dot{q}_j &= \tau_j \cdot (r_j -r_{j-1})-\frac{|\tau_{j+1}-\tau_j|^2 }{(q_j + q_{j+1})}  -\frac{|\tau_{j-1}-\tau_j|^2 }{(q_j + q_{j-1})} \label{(4.6)} \\
&= \tau_j \cdot (r_j -r_{j-1})-\frac{(q_j + q_{j+1})}{4}|\dot{u}_{j}-r_j|^2 - \frac{(q_j + q_{j-1})}{4}|\dot{u}_{j-1}-r_{j-1}|^2 \,. \label{(4.6)bis}
\end{align}
 \item Furthermore with a constant $C>0$
\begin{equation} \label{eq:disel}
\begin{split}
\max_{1 \leq j \leq N} q_{j}(t) &\leq C \max_{1 \leq j \leq N} q_{j}(0), \\
\| u_{hx}(\cdot,t) \|_{L^\infty(S^1)} &\leq C \| u_{h0x}(\cdot) \|_{L^\infty(S^1)}.
\end{split}
\end{equation}
 \item Moreover, there is a $C>0$ such that 
\begin{equation} \label{eq:Ah1}
\int_0^{\intertime} (q_j + q_{j+1}) |\dot{u}_j-r_j|^2 dt \leq C h, \quad \int_0^{\intertime} \frac{|\tau_{j\pm1}-\tau_{j}|^{2}}{(q_{j}+q_{j\pm1})} dt \leq C h.
\end{equation}
\end{enumerate}
\end{lemma}

\begin{proof}
From the definition of $q_j$ we obtain by differentiating in time
\[ 
\dot{q}_j= \tau_j \cdot (\dot{u}_j - \dot{u}_{j-1})\,.
\] 
From the system \eqref{eq:odegeoB} together with $\tau_j \cdot(\tau_{j+1}-\tau_j)=-\frac{1}{2}|\tau_{j+1}-\tau_j|^2$ we 
infer that
\[ 
\tau_j \cdot \dot{u}_j =\frac{2}{q_j + q_{j+1}} \tau_j \cdot \bigl(\tau_{j+1}-\tau_j\bigr) + \tau_j \cdot r_j
=-\frac{|\tau_{j+1}-\tau_j|^2}{q_j + q_{j+1}} + \tau_j \cdot r_j \,.
\] 
Arguing similarly for the term $\tau_j \cdot \dot{u}_{j-1}$ one obtains equation \eqref{(4.6)}. Using \eqref{eq:odegeo} we can write
$\tau_{j+1}-\tau_j= \frac{q_j+q_{j+1}}{2} (\dot{u}_j-r_j)$ and \eqref{(4.6)bis} follows which proves the first assertion.

For the second assertion we set $f_j := f(c_j)$ for simplicity. Note that by \eqref{nubar}
\[ 
\tau_{j}\cdot(r_{j}-r_{j-1}) = f_{j}\frac{q_{j+1}}{q_{j}+q_{j+1}}( \tau_{j} \cdot \nu_{j+1}) - f_{j-1}\frac{q_{j-1}}{q_{j}+q_{j-1}}( \tau_{j} \cdot \nu_{j-1}). 
\] 
Since $\frac{q_{j \pm 1}}{q_{j}+q_{j \pm 1}} \leq 1$ we get that
\begin{multline*}
\Big{|} \frac{q_{j+1}}{q_{j} + q_{j+1}} f_{j} \tau_{j} \cdot \nu_{j+1} \Big{|} = \Big{|} \frac{q_{j+1}}{q_{j} + q_{j+1}} \sqrt{q_{j} + q_{j+1}} f_{j} \tau_{j} \cdot \frac{\nu_{j+1} - \nu_{j}}{\sqrt{q_{j} + q_{j+1}}} \Big{|} \\
\leq \frac{1}{2\epsilon} \frac{q_{j+1}^{2}}{q_{j} + q_{j+1}} | f_{j} |^{2} + \frac{\epsilon}{2} \frac{| \nu_{j+1} - \nu_{j} |^{2}}{q_{j} + q_{j+1}} \leq \frac{1}{2\epsilon} q_{j+1} | f_{j} |^{2} + \frac{\epsilon}{2} \frac{| \tau_{j+1} - \tau_{j} |^{2}}{q_{j} + q_{j+1}}
\end{multline*}
and, similarly,
\[
\Big{|} \frac{q_{j-1}}{q_{j} + q_{j-1}} f_{j-1} \tau_{j} \cdot \nu_{j-1} \Big{|} \leq \frac{1}{2\epsilon} q_{j-1} | f_{j-1} |^{2} + \frac{\epsilon}{2} \frac{| \tau_{j-1} - \tau_{j} |^{2}}{q_{j} + q_{j-1}}.
\]
Therefore
\begin{equation} 
\label{r-r}
|\tau_{j}\cdot(r_{j}-r_{j-1})| \leq \frac{1}{2\epsilon} \big{(} q_{j+1} | f_{j} |^{2} + q_{j-1} | f_{j-1} |^{2} \big{)} + \frac{\epsilon}{2} \frac{| \tau_{j+1} - \tau_{j} |^{2}}{q_{j} + q_{j+1}} + \frac{\epsilon}{2} \frac{| \tau_{j-1} - \tau_{j} |^{2}}{q_{j} + q_{j-1}}
\end{equation} 
Equation \eqref{(4.6)} and equation \eqref{r-r} with $\epsilon =1$ yield that
\begin{align*}
\dot{q}_j &\leq |\tau_j \cdot (r_j -r_{j-1})| -\frac{|\tau_{j+1}-\tau_j|^2 }{(q_j + q_{j+1})}  -\frac{|\tau_{j-1}-\tau_j|^2 }{(q_j + q_{j-1})} \\
&\leq \|f\|^{2}_{L^{\infty}(\R)} \frac{1}{2} \big{(} q_{j-1} + q_{j+1} \big{)}.
\end{align*}
Integrating with respect to $t$ we infer that
\[
 q_{j}(t) \leq \max_{1 \leq i \leq N} q_{i}(t) \leq \max_{1 \leq i \leq N} q_{i}(0) + C \int_{0}^{t} \max_{1 \leq i \leq N} q_{i}(t') dt' \,.
\]
Applying a Gronwall argument we obtain the first estimate of \eqref{eq:disel}. The second one is a direct consequence of the first one thanks to \eqref{(4.1)}. 

From \eqref{(4.6)bis} and \eqref{r-r} we infer that
\begin{align*}
&\frac{(q_j + q_{j+1})}{4} |\dot{u}_j-r_j|^2 + \frac{(q_j + q_{j-1})}{4} |\dot{u}_{j-1}-r_{j-1}|^2 \\
& \leq | \tau_j \cdot (r_j -r_{j-1}) |- \dot{q}_j \\
& \leq \epsilon \frac{|\tau_{j+1}-\tau_{j}|^{2}}{(q_{j}+q_{j+1})} + \epsilon \frac{|\tau_{j-1}-\tau_{j}|^{2}}{(q_{j}+q_{j-1})} + C_{\epsilon} \|f\|_{L^\infty(\R)}^{2} (q_{j-1}+q_{j+1}) - \dot{q}_j \\
& = \epsilon \frac{(q_j + q_{j+1})}{4} |\dot{u}_j-r_j|^2 + \epsilon \frac{(q_j + q_{j-1})}{4} |\dot{u}_{j-1}-r_{j-1}|^2 + C_{\epsilon} \|f\|_{L^\infty(\R)}^{2} (q_{j-1}+q_{j+1}) - \dot{q}_j,
\end{align*}
where we have used \eqref{eq:odegeoB} in the last equality. Choosing $\epsilon$ appropriately, integrating with respect to time, and using that $q_{j}(t) = h_{j}|u_{hx}|_{|_{S_{j}}} \leq Ch$ thanks to \eqref{eq:disel}, we obtain the estimates \eqref{eq:Ah1}.
\end{proof}

%%%%%%%%%%%%%%%%%%%%%
\section{Error estimates}
\label{sec:errestim}

In this section we prove some estimates that will enable to show convergence of the semi-discrete solutions $(u_h, c_h)$ of \eqref{eq:wpgeoh}, \eqref{eq:wpspdeh} to the solution $(u,c)$ of the continuous problem as specified in Assumption \ref{ass:sol}. For this purpose let us assume that for $h>0$ there is a unique solution $(u_h, c_h)$ for $t \in [0,{\intertime}]$ with some ${\intertime} \in (0,T]$. 

We commence with some calculations for the error of the length element $| u_{x} | - | u_{hx} |$ and show some preliminary estimates in Lemma \ref{lem:err_est_A}. These are used to obtain an estimate of $c - c_h$ in suitable norms, see Lemma \ref{lem:err_est_B}. An estimate of $u - u_h$ in suitable norms (see Lemma \ref{lem:err_est_C}) follows the lines of \cite{D99} and involves an integral term of the error of the length element which we estimate last in Lemma \ref{lem:err_est_D}.

We will use the abbreviations
\begin{align*}
q=|u_x|\,, \qquad q_j = h_j |u_{hx}|_{|_{S_j}}\,, \qquad q_h = |u_{hx}|\,.
\end{align*}
Recalling \eqref{(3.5)} and \eqref{(4.6)bis}, we can write for each grid element $S_j=[x_{j-1},x_j]$ the following equation:
\begin{align}
(h_jq-q_j)_t 
%&= -h_j q\, u_t \cdot ( u_t -r) + \frac{1}{4}(q_j + q_{j+1}) |\dot{u}_j-r_j |^2 \nonumber \\
%&\quad +\frac{1}{4}(q_j + q_{j-1}) |\dot{u}_{j-1} -r_{j-1}|^2 -\tau_j \cdot(r_j -r_{j-1}) \nonumber \\
&=-\Bigl\{ \frac{1}{2}h_j q\, |u_t -r|^2  - \frac{1}{4}(q_j + q_{j+1}) |\dot{u}_j-r_j |^2\Bigr \} \nonumber \\
& \quad -\Bigl\{ \frac{1}{2}h_j q\, |u_t-r|^2 - \frac{1}{4}(q_j + q_{j-1}) |\dot{u}_{j-1} -r_{j-1}|^2 \Bigr \} \nonumber \\
& \quad - \Bigl( h_j q \,(u_t -r) \cdot r + \tau_j \cdot(r_j -r_{j-1}) \Bigr) \nonumber \\
&=-B^+ - B^- - \hat{B}\,. \label{eq:splitB}
\end{align}
Using \eqref{nubar} we can write
\begin{align*} 
\hat{B} %&= h_j q \,(u_t -r) \cdot r + \tau_j \cdot(r_j -r_{j-1})\\
& = \frac{1}{2} h_{j} q \, (u_{t} -r) \cdot \nu f(c) + f(c_{j}) \frac{q_{j+1}}{q_{j}+q_{j+1}} \tau_{j} \cdot \nu_{j+1} \\
& \quad + \frac{1}{2} h_{j} q \, (u_{t} -r) \cdot \nu f(c) - f(c_{j-1}) \frac{q_{j-1}}{q_{j}+q_{j-1}} \tau_{j} \cdot \nu_{j-1}= \hat{B}_{1}+ \hat{B}_{2}.
\end{align*}
Observe that
\[ %\begin{multline*}
\frac{2}{q_{j}+q_{j+1}} \tau_{j} \cdot \nu_{j+1} %= \frac{2}{q_{j}+q_{j+1}} \tau_{j} \cdot (\nu_{j+1}-\nu_{j}) \\
= \tau_{j} \cdot \frac{2 (\tau_{j+1}-\tau_{j})^{\perp}}{q_{j}+q_{j+1}} = -\nu_{j} \cdot \frac{2 (\tau_{j+1}-\tau_{j})}{q_{j}+q_{j+1}}= \nu_{j} \cdot (r_{j}-\dot{u}_{j}) 
\] %\end{multline*}
by \eqref{eq:odegeoB}, so we can write
\begin{align}
\hat{B}_{1} %&= \frac{1}{2} h_{j} q \, (u_{t} -r) \cdot \nu f(c) + f(c_{j}) \frac{q_{j+1}}{q_{j}+q_{j+1}} \tau_{j} \cdot \nu_{j+1} \nonumber \\
&= \frac{1}{2} h_{j} q (f(c)-f(c_{j}))\, (u_{t} -r) \cdot \nu + \frac{1}{2}f(c_{j}) (h_{j}q -q_{j+1}) \, (u_{t} -r) \cdot \nu \nonumber \\
& \quad + \frac{1}{2} f(c_{j}) q_{j+1} (u_{t} -r) \cdot (\nu -\nu_{j} ) + \frac{1}{2} f(c_{j}) q_{j+1} \nu_{j} \cdot [(u_{t} -r) - (\dot{u}_{j} -r_{j})]. \label{eq:Bh1}
\end{align}
Similarly one can show that $\frac{2}{q_{j}+q_{j-1}} \tau_{j} \cdot \nu_{j-1} = -\nu_{j} \cdot (r_{j-1} -\dot{u}_{j-1})$ whence
\begin{align}
\hat{B}_{2} %&= \frac{1}{2} h_{j} q \, (u_{t} -r) \cdot \nu f(c) - f(c_{j-1}) \frac{q_{j-1}}{q_{j}+q_{j-1}} \tau_{j} \cdot \nu_{j-1} \nonumber \\
& = \frac{1}{2} h_{j} q (f(c)-f(c_{j-1})) \, (u_{t} -r) \cdot \nu + \frac{1}{2} f(c_{j-1}) (h_{j}q -q_{j-1}) (u_{t} -r) \cdot \nu \nonumber \\
& \quad + \frac{1}{2} f(c_{j-1}) q_{j-1} (u_{t} -r) \cdot (\nu -\nu_{j}) + \frac{1}{2} f(c_{j-1}) q_{j-1} \nu_{j} \cdot [(u_{t} -r) - (\dot{u}_{j-1} -r_{j-1})]. \label{eq:Bh2}
\end{align}
Let us also set
\begin{align} 
B^+ %&= \frac{1}{2}h_j q\, |u_t -r|^2  - \frac{1}{4}(q_j + q_{j+1}) |\dot{u}_j-r_j |^2 \notag \\
&= \frac{1}{4} |u_t -r|^2\, (q h_j-q_j) + \frac{1}{4} |u_t -r|^2\, (q h_j-q_{j+1}) \notag \\
&\quad + \frac{1}{4}(q_j + q_{j+1})\bigl( |u_t -r|^2 -|\dot{u}_j -r_j|^2\bigr) \notag \displaybreak[0]\\
%&=\frac{1}{4} |u_t -r|^2\, (q h_j-q_j) +\frac{1}{4} |u_t -r|^2\, (q h_j-q_{j+1}) \notag \\
% &\quad + \frac{1}{4}(q_j + q_{j+1})\Bigl\{ [(u_t -r)-(\dot{u}_j -r_j) ] \cdot(u_t -r) + (\dot{u}_j -r_j) \cdot [(u_t -r)-(\dot{u}_j -r_j)] \Bigr \} \notag \\
&=\frac{1}{4} |u_t -r|^2\, (q h_j-q_j) + \frac{1}{4} |u_t -r|^2\, (q h_j-q_{j+1}) \notag \\
& \quad + \frac{1}{4}(q_j + q_{j+1}) (\dot{u}_j -r_j) \cdot [(u_t -r)-(\dot{u}_j -r_j)] \notag \\
& \quad + \frac{1}{4}(q_j + q_{j+1}) [(u_t -r)-(\dot{u}_j -r_j) ] \cdot(u_t -r) = B^+_1 +B^+_2 +B^+_3 +B^+_4 \,. \label{eq:Bplus}
\end{align}

%%%%% first estimate length element
\begin{lemma} \label{lem:err_est_A}
Assume that ${\intertime} \in (0,T]$ is such that 
\[
 \frac{C^{**}}{2} \leq |u_{hx}| \leq 2 C^{*} \quad \mbox{on } [0,{\intertime}]. 
\] 
Then there exists a constant $C$ 
%that depends on $\bar{C}$, $C^*$, $C^{**}$, $\tilde{C}_{1}$, $\|f\|_{L^{\infty}(\R)}$, $\|f'\|_{L^{\infty}(\R)}$
such that for any time $t \in (0,{\intertime}]$ we have:
\begin{enumerate}
\item On each $S_{j}$ we can write
\[ %\begin{align}\label{expL1}
||u_{x}|_{t}- |u_{hx}|_{t}| \leq C \Big (|u_{t}-\dot{u}_{j}|^{2} + |u_{t}-\dot{u}_{j-1}|^{2} \Big) + C L_{j},
\] %\end{align}
where
\begin{align*}
L_{j} &:=|c-c_j| + |c-c_{j-1}| + |\tau -\tau_j| + |\tau -\tau_{j-1}| + |\tau- \tau_{j+1}| \\
&\quad + |q-\frac{q_{j}}{h_{j}}|+ |q-\frac{q_{j+1}}{h_{j+1}}| + |q-\frac{q_{j-1}}{h_{j-1}}| + |u_t - \dot{u}_j|+ |u_t - \dot{u}_{j-1}| + h .
\end{align*}
\item Moreover
\begin{equation} \label{L2ab}
\sum_{j=1}^{N}\int_{S_{j}}|u_{t}-\dot{u}_{j}|^{2} + |u_{t}-\dot{u}_{j-1}|^{2} dx \leq C h^{2} + C \int_{S^1} |u_t -u_{ht}|^2 dx
\end{equation}
and 
\begin{multline} \label{L1ab}
\sum_{j=1}^{N}\int_{S_{j}}|L_{j}|^{2} dx \leq C \int_{S^1}|c-c_h|^2 dx + C \int_{S^1} |\tau-\tau_h|^2 dx \\
+ C \int_{S^1} (|u_{hx}|-|u_x|)^2 dx + C \int_{S^1} |u_t -u_{ht}|^2 dx + C h^2.
\end{multline}
\end{enumerate}
\end{lemma}

\begin{proof}
As we have assumed that $2C^* \geq q_h \geq C^{**} /2$, the discrete length elements are comparable, in other words
\begin{equation}\label{(27)}
C^{-1} q_{j+1} \leq q_j \leq C q_{j+1}\,.
\end{equation}

Note that $|f(c)-f(c_j)| \leq \| f' \|_{L^\infty(\R)}|c-c_j|$ and 
\begin{equation}\label{argQ}
|q-\frac{q_{j+1}}{h_j}| =|q-\frac{q_{j+1}}{h_{j+1}} + q_{j+1}(\frac{1}{h_{j+1}}-\frac{1}{h_j})| \leq | q-\frac{q_{j+1}}{h_{j+1}}| + \frac{q_{j+1}}{h_{j+1}}\frac{|h_j -h_{j+1}|}{h_j} \leq | q-\frac{q_{j+1}}{h_{j+1}}| +Ch
\end{equation}
(which follows by \eqref{(4.1)}). Thus, using \eqref{Cstar}, %\eqref{(3.5)},
 \eqref{(27)}, and the bound $|u_{hx}| \leq 2C^*$ we obtain from \eqref{eq:Bh1}, \eqref{eq:Bh2} for some $C>0$ that
\begin{align*}
\frac{|\hat{B}|}{h_j} &\leq C\|u_{t}-r \|_{L^{\infty}} \Big(h+ |c-c_j| + |c-c_{j-1}|+ |\tau- \tau_j| + |q-\frac{q_{j+1}}{h_{j+1}}| + |q-\frac{q_{j-1}}{h_{j-1}}| \Big) \\
& \quad + C \Big ( | r - r_{j-1}| +  |r-r_{j}| + |u_t - \dot{u}_j|+ |u_t - \dot{u}_{j-1}| \Big).
\end{align*}
Observe that on $S_j$ 
\begin{align} \label{r-rj}
|r-r_j| %&=|f(c)\nu - f(c_j) \bar{\nu}_j| \notag \\
&\leq |(f(c)-f(c_j)) \nu|+ |f(c_j)(\nu-\bar{\nu}_j)| \notag \\
&\leq \|f'\|_{L^\infty(\R)} |c-c_j| + \|f\|_{L^\infty (\R)}\frac{q_j}{q_j+q_{j+1}}
|\nu- \nu_j|+ \|f\|_{L^\infty(\R)}\frac{q_{j+1}}{q_j +q_{j+1}}|\nu -\nu_{j+1}| \notag \\
& \leq C |c-c_j| + C (|\tau -\tau_j| + |\tau- \tau_{j+1}|),
\end{align}
and similarly for $|r - r_{j-1}|$. Hence we get
\begin{align*}
\frac{|\hat{B}|}{h_j} &\leq C \left( |c-c_j| + |c-c_{j-1}| + |\tau -\tau_j| + |\tau -\tau_{j-1}| + |\tau- \tau_{j+1}| \right)\\
&+ C \left( |q-\frac{q_{j+1}}{h_{j+1}}| + |q-\frac{q_{j-1}}{h_{j-1}}| \right) + C \left (|u_t - \dot{u}_j|+ |u_t - \dot{u}_{j-1}| \right) + C h.
\end{align*}

Note that $B_{3}^{+}$ defined in \eqref{eq:Bplus} can be written as
\begin{align*}
B_{3}^{+}= \frac{q_j + q_{j+1}}{4} (\dot{u}_j -u_{t}) \cdot [(u_t -\dot{u}_j) - (r -r_j)]  + \frac{q_j + q_{j+1}}{4} (u_{t} -r_j) \cdot [(u_t -\dot{u}_j) - ( r -r_j)]. 
\end{align*}
Using the $L^{\infty}$-bounds for $u_{t}$, $r$ and $r_{j}$ (recall \eqref{eq:r}, \eqref{eq:defr_j}, and $|\bar{\nu}_{j}| \leq 1$), \eqref{(27)}, the bound $|u_{hx}|\leq 2 C^{*}$, embedding theory, and arguments similar to those employed in \eqref{argQ}, and \eqref{r-rj}, we infer that
\begin{align*}
\frac{|B^+|}{h_j} &\leq C h + C |u_{t}-\dot{u}_{j}|^{2} \\
& \quad+ C \left( |q-\frac{q_{j+1}}{h_{j+1}}| + |q-\frac{q_{j}}{h_{j}}| + |u_t - \dot{u}_j| + |c-c_j| +  |\tau -\tau_j| + |\tau- \tau_{j+1}|
\right).
\end{align*}
Arguing similarly for $B^-$, and putting all estimates together we finally obtain from \eqref{eq:splitB} that
\begin{align*}
|\dot{q}-\frac{\dot{q}_j}{h_j}| & \leq C \left( \frac{|\hat{B}|}{h_{j}}  +
\frac{|B^{+}|}{h_{j}} + \frac{|B^{-}|}{h_{j}}
\right)\\
& \leq C \Big (|u_{t}-\dot{u}_{j}|^{2} + |u_{t}-\dot{u}_{j-1}|^{2} \Big)\\
& \quad + C \Big( |c-c_j| + |c-c_{j-1}| + |\tau -\tau_j| + |\tau -\tau_{j-1}| + |\tau- \tau_{j+1}| \\
& \quad \quad \quad + |q-\frac{q_{j}}{h_{j}}|+ |q-\frac{q_{j+1}}{h_{j+1}}| + |q-\frac{q_{j-1}}{h_{j-1}}| + |u_t - \dot{u}_j|+ |u_t - \dot{u}_{j-1}| +h \Big ) %\\
%&=:C \Big (|u_{t}-\dot{u}_{j}|^{2} + |u_{t}-\dot{u}_{j-1}|^{2} \Big) + C L_{j}
\end{align*}
which shows the first claim.

As $u_{ht_{|_{S_j}}}=\dot{u}_{j-1} + (\dot{u}_j - \dot{u}_{j-1})\frac{(x-x_{j-1})}{h_j}$, we have that $u_{ht}(x_j)=\dot{u}_j$. On the other hand $I_h u_t (x_j)= u_t(x_j)$. Therefore for $x\in S_j$ we can write 
\begin{multline*}
u_t(x)-\dot{u}_j = u_t (x) - u_t(x_j) + I_h u_t (x_j) - u_{ht} (x_j) \\ 
= \int_{x_j}^x u_{tx} (\xi) \,d\xi + I_h u_t (x_j) - u_{ht} (x_j) \leq \sqrt{h} \|u_t\|_{H^1(S_j)} +I_h u_t (x_j) - u_{ht} (x_j)\,.
\end{multline*}
For $w_h (x): = I_h u_t (x) - u_{ht} (x)$ we can use the inverse estimate \eqref{IE-2}.
%$\| w_h \|_{L^{\infty}(S_j)} \leq \frac{C}{\sqrt{h_j}} \|w_h \|_{L^2 (S_j)}$.
Therefore 
\begin{align*}
|u_t(x)-\dot{u}_j |^2 &\leq C h \|u_t\|_{H^1(S_j)}^2 +\frac{C}{h_j} \int_{S_j} (I_h u_t -u_{ht})^2(\xi) \, d\xi \\
&\leq C h \|u_t\|_{H^1(S_j)}^2 +\frac{C}{h_j} \int_{S_j} 
(I_h u_t -u_t )^2 (\xi)\, d\xi + \frac{C}{h_j}\int_{S_j} ( u_t -u_{ht})^2(\xi) \, d\xi \\
& \leq Ch \|u_t\|_{H^1(S_j)}^2 +\frac{C}{h_j} \int_{S_j} ( u_t -u_{ht})^2(\xi) \, d\xi
\end{align*}
by \eqref{(4.1)} and \eqref{(4.2)}. Arguing similarly for the term $|u_t - \dot{u}_{j-1}|$,
integrating, and summing up over the grid intervals we obtain \eqref{L2ab}. 

Regarding the last estimate, observe that for any $y\in S_{j+1}$ and $x\in S_j$ we can write 
\begin{align*}
|q(x) -\frac{q_{j+1}}{h_{j+1}}| & \leq |q(x)-q(y)| + | q(y)- \frac{q_{j+1}}{h_{j+1}}| \leq C\sqrt{h} \| u\|_{H^2 (S_j \cup S_{j+1})} + | q(y)- \frac{q_{j+1}}{h_{j+1}}|.
\end{align*}
Thanks to the continuity of $q$ we can choose $y \in S_{j+1}$ such that 
\begin{equation*}
h_{j+1} (q(y)-q_{h_{|_{S_{j+1}}}})^2 \leq \int_{S_{j+1}} (q(\xi)-q_{h_{|_{S_{j+1}}}})^2 d\xi = \int_{S_{j+1}} (q-q_h )^2 dx\,. 
\end{equation*}
Using this fact and \eqref{(4.1)} yields that
\begin{align*}
\int_{S_j} |q-\frac{q_{j+1}}{h_{j+1}}|^2 dx \leq C h^2 \| u\|_{H^2 (S_j \cup S_{j+1})}^2 + C \int_{S_{j+1}} (q-q_h)^2 dx.
\end{align*}
With similar arguments for $q - q_{j}/h_{j}$ and $q - q_{j-1}/h_{j-1}$ we obtain that 
\begin{multline}\label{bella0}
\int_{S_j} |q-\frac{q_{j}}{h_{j}}|^2+ |q-\frac{q_{j+1}}{h_{j+1}}|^2 + |q-\frac{q_{j-1}}{h_{j-1}}|^2 dx \\
\leq Ch^2 \| u\|_{H^2 (S_j \cup S_{j+1} \cup S_{j-1})}^2 + C \int_{S_j \cup S_{j+1} \cup S_{j-1}} |q-q_h|^2 dx.
\end{multline}
The terms $|c-c_j|$ and $|c-c_{j-1}|$ can be estimated similarly as $|u_{t} - \dot{u}_{j}|$ and $|u_{t} - \dot{u}_{j-1}|$ whence 
\begin{multline}\label{bella1}
\int_{S_j} |c-c_j|^2 + |c-c_{j-1}|^2 + |u_t - \dot{u}_{j-1}|^2 + |u_t-\dot{u}_j |^2  \, dx\\
\leq Ch^2 \|u_t \|_{H^1 (S_j)}^{2} + C \int_{S_j} ( u_t -u_{ht})^2(\xi) \, d\xi + Ch^2 \|c \|_{H^1 (S_j)}^2 + C \int_{S_j} |c-c_h|^2 dx. 
\end{multline}
We can use the boundedness of $|u_x|$ from below to get for any $x$, $y \in S_j \cup S_{j+1}$ 
(suppose $y\leq x$ or change the order of integration otherwise) 
$$|\tau (x)-\tau (y)| \leq \int_y^x |\tau_{x}(\xi)|\, d\xi \leq C\sqrt{2h}(\int_{S_j \cup S_{j+1}} |u_{xx}(\xi)|^2\, d\xi )^{1/2}\,.$$
Choosing $y \in S_{j+1}$ such that $h_{j+1}|\tau(y)-\tau_{j+1}|^2 \leq \int_{S_{j+1}} (\tau-\tau_h)^2 dx$ we can write
\begin{align*}
|\tau(x)-\tau_{j+1}|^2 &\leq |\tau(x)-\tau(y)|^2 + |\tau(y)-\tau_{j+1}|^2 \\
&\leq Ch \| u\|_{H^2(S_j \cup S_{j+1})}^2 + \frac{C}{h} \int_{S_{j+1}} (\tau-\tau_h)^2 dx\,.
\end{align*}
Repeating the same sort of argument for $|\tau- \tau_{j-1}|$ and integrating over $S_j$ we get
\begin{multline}\label{bella2}
\int_{S_j}|\tau-\tau_{j+1}|^2 + |\tau -\tau_j|^2 + |\tau-\tau_{j-1}|^{2} dx \\
\leq Ch^2 \| u\|_{H^2(S_j \cup S_{j+1} \cup S_{j-1})}^2 + C \int_{S_j \cup S_{j+1} \cup S_{j-1}} (\tau-\tau_h)^2 dx \,.
\end{multline}
Putting all estimates together and summing up over the grid intervals \eqref{L1ab} follows.
\end{proof}

%%%%% estimate c-c_{h}
\begin{lemma} \label{lem:err_est_B}
Assume that ${\intertime} \in (0,T]$ is such that
\begin{equation} \label{assu}
\begin{split}
\frac{C^{**}}{2} &\leq |u_{hx}| \leq 2 C^{*} \mbox{ on } [0,{\intertime}], \mbox{ and} \\
\|c_{h}\|_{C([0,{\intertime}],L^\infty(S^{1}))} &\leq 2 \embconst \|c\|_{C([0,T],H^{1}(S^1))} %C(S^{1})\|c\|_{C([0,T],H^{1}(S^1))}.
\end{split}
\end{equation}
where $\embconst$ is a constant for the embedding $H^1(S^1) \hookrightarrow L^\infty(S^1)$. Then the following estimate holds with some constant $C>0$:
\begin{align}
\int_{S^1} |c({\intertime}) - & c_{h}({\intertime})|^2 \, dx + \int_0^{\intertime} \int_{S^1} |c_x - c_{hx}|^2 \, dx dt \nonumber \\
&\leq C \int_{S^1} \big{(} |u_x({\intertime})| - |u_{hx}({\intertime})| \big{)}^2 \, dx + C h^{2} \displaybreak[0] \nonumber \displaybreak[0] \\
&\quad + C\int_{0}^{{\intertime}} \int_{S^1} |c - c_h|^2 \, dx dt + C\int_{0}^{{\intertime}} \int_{S^{1}}|u_{t}-u_{ht}|^{2 } dx dt \nonumber \\
&\quad + C \int_{0}^{{\intertime}} \int_{S^{1}}|\tau-\tau_{h}|^{2 } dx dt + C \int_{0}^{{\intertime}} \int_{S^1} \big{(} |u_x| - |u_{hx}| \big{)}^2 \, dx dt. \label{c-c_h}
\end{align}
\end{lemma} 

\begin{proof}
The difference between the continuous \eqref{eq:wpspde_zt} and the discrete version \eqref{eq:wpspde_zth} reads
\[
\int_{S^1} (c |u_x| - c_{h} |u_{hx}|)_t \zeta_{h} \, dx + \int_{S^1} \Big{(} \frac{c_x}{|u_x|} - \frac{c_{hx}}{|u_{hx}|} \Big{)} \zeta_{hx} \, dx = 0
\]
for all test functions $\zeta_h(x,t)$ of the form $\zeta_h = \sum_j \zeta_j(t) \varphi_j(x)$. Choosing 
\[
 \zeta_h = I_h(c) - c_h = c - c_h + I_h(c) - c
\]
a short calculation yields that
\begin{align}
\frac{d}{dt} \Big{(} \int_{S^1} & \frac{1}{2} (c - c_{h})^2 |u_{hx}| \, dx \Big{)} + \int_{S^1} \frac{|(c - c_h)_x|^2}{|u_{hx}|} \, dx \nonumber \\
&= \int_{S^1} \big{(} c (|u_{hx}| - |u_x|) \big{)}_t (c - c_h) \, dx - \int_{S^1} \frac{1}{2} (c - c_h)^2 |u_{hx}|_t \, dx \nonumber \displaybreak[0] \\
&+ \frac{d}{dt} \Big{(} \int_{S^1} (c |u_x| - c_{h} |u_{hx}|) (c - I_h(c)) \, dx \Big{)} \displaybreak[0] \nonumber \\
&- \int_{S^1} (c |u_x| - c_{h} |u_{hx}|) \big{(} c - I_h(c) \big{)}_t \, dx \displaybreak[0] \nonumber \\
&+ \int_{S^1} \frac{(c - c_h)_x (c - I_h(c))_x}{|u_{hx}|} \, dx \displaybreak[0] \nonumber \\
&+ \int_{S^1} c_x \frac{(c - c_h)_x}{\sqrt{|u_{hx}|}} \frac{|u_x| - |u_{hx}|}{\sqrt{|u_{hx}|} |u_x|} \, dx + \int_{S^1} c_x (I_h(c) - c)_x \frac{|u_x| - |u_{hx}|}{|u_{hx}| \, |u_x|} \, dx \nonumber \\
&= \sum_{j=1}^{7} K_j \label{eq:toest_c}.
\end{align}

Using Lemma \ref{lem:err_est_A} we can write
\begin{align*}
|K_1| &= \left| \int_{S^1} c_t (|u_{hx}| - |u_x|) (c - c_h) \, dx + \int_{S^1} c (|u_{hx}| - |u_x|)_t (c - c_h) \, dx \right| \\
& \leq \left| \int_{S^1} c_t (|u_{hx}| - |u_x|) (c - c_h) \, dx \right| \\
& \quad + C \sum_{j=1}^{N}\int_{S_{j}} |c||c-c_{h}|(|u_{t}-\dot{u}_{j}|^{2} + |u_{t}-\dot{u}_{j-1}|^{2}) dx + C \sum_{j=1}^{N}\int_{S_{j}} |c||c-c_{h}| |L_{j}| dx \displaybreak[0] \\
& \leq \frac{1}{2}\|c_t\|_{L^\infty(S^1)}^{2} \int_{S^1} (|u_{hx}| - |u_x|)^2 \, dx + \frac{1}{2} \int_{S^1} (c - c_h)^2 \, dx \displaybreak[0] \\
& \quad + C \| c\|_{L^{\infty}(S^{1})} \| c-c_{h}\|_{L^{\infty}(S^{1})} \Big{(} h^{2} + \int_{S^{1}} |u_{t}-u_{ht}|^{2} dx \Big{)} \\
& \quad + C \int_{S^{1}} |c|^{2} |c-c_{h}|^{2} dx + C \sum_{j=1}^{N}\int_{S_{j}} |L_{j}|^{2} dx.
\end{align*}
Together with \eqref{cbounds}, the assumptions \eqref{assu}, and \eqref{L1ab} we obtain that 
\begin{align*}
|K_{1}|& \leq C \int_{S^1} (|u_{hx}| - |u_x|)^2 \, dx + C \int_{S^1} (c - c_h)^2 |u_{hx}| \, dx \\
& \quad + C \int_{S^{1}} |u_{t}-u_{ht}|^{2} dx + C \int_{S^1} |\tau-\tau_h|^2 \, dx + C h^{2}.
\end{align*}
Similarly for $K_{2}$, using again Lemma \ref{lem:err_est_A}, \eqref{cbounds}, embedding theory and the assumptions \eqref{assu} to estimate $\| c - c_h \|_{L^\infty(S^1)}$ we can write 
\begin{align*}
|K_{2}| & \leq \frac{1}{2} \int_{S^{1}}|c-c_{h}|^{2} ||u_{hx}|_{t} - |u_{x}|_{t}| dx + \frac{1}{2} \int_{S^{1}}|c-c_{h}|^{2} ||u_{x}|_{t}| dx \displaybreak[0] \\
& \leq C \|c-c_{h} \|_{L^{\infty}(S^{1})}^{2} \sum_{j=1}^{N} \int_{S_{j}} (|u_{t}-\dot{u}_{j}|^{2} + |u_{t}-\dot{u}_{j-1}|^{2}) dx \\
& \quad + C \|c-c_{h} \|_{L^{\infty}(S^{1})} \sum_{j=1}^{N} \Big{(} \frac{1}{2} \int_{S_{j}} |c-c_{h}|^2 dx + \frac{1}{2} \int_{S_{j}} |L_{j}|^2 dx \Big{)} \\
& \quad + C \int_{S^{1}}|c-c_{h}|^{2} dx \displaybreak[0] \\
& \leq C h^{2} + C \int_{S^{1}}|u_{t}-u_{ht}|^{2} dx + C \int_{S^{1}} |c-c_{h}|^{2} dx \\
& \quad + C \int_{S^1} (|u_{hx}|-|u_x|)^2 dx + \int_{S^1} |\tau-\tau_h|^2 dx.
\end{align*} 
For $K_{3}$ we note that by \eqref{cbounds}, \eqref{assu}, and \eqref{(4.2)}
\begin{align}
& \Big |\int_{S^1} (c |u_x| - c_{h} |u_{hx}|) (c - I_h(c)) \, dx \Big | \nonumber \\
& = \Big | \int_{S^1} (c - c_h) |u_{hx}| (c - I_h(c)) \, dx + \int_{S^1} c (|u_x| - |u_{hx}|) (c - I_h(c)) \, dx \Big | \nonumber \\
& \leq \hat{\eps} \int_{S^1} (c - c_h)^2 |u_{hx}| \, dx + C \int_{S^1} (|u_x| - |u_{hx}|)^2 \, dx + C_{\hat{\eps}} h^2 \| c \|_{H^1(S^1)}^2 \label{eq:estK3}
\end{align}
with $\hat{\eps}>0$ that will be picked later on. We will refer to this estimate later on when integrating \eqref{eq:toest_c} with respect to time. 
For the term $K_{4}$ we infer from \eqref{(4.2)} and \eqref{assu} that
\begin{align*}
|K_4| &= \Big{|} \int_{S^1} c (|u_x| - |u_{hx}|) (c_t - I_h(c_t)) \, dx + \int_{S^1} (c - c_{h}) (c_t - I_h(c_t)) |u_{hx}| \, dx \Big{|} \\
&\leq C \int_{S^1} (|u_x| - |u_{hx}|)^2 \, dx + C \int_{S^1} (c - c_h)^2 | u_{hx} | \, dx  + C \|c_t\|_{H^1(S^1)}^{2} h^2.
\end{align*}
By the interpolation estimates \eqref{(4.2)}, \eqref{(4.2)bis}, \eqref{assu}, and embedding theory we have the following estimates for the terms involving spatial gradients (for $\epsilon > 0$ arbitrarily small):
\begin{align*}
|K_5| &\leq \epsilon \int_{S^1} \frac{|(c - c_h)_x|^2}{|u_{hx}|} \, dx + C_{\epsilon} \int_{S^1} \frac{|(c - I_h(c))_x|^2}{|u_{hx}|} \, dx \\
&\leq \epsilon \int_{S^1} \frac{|(c - c_h)_x|^2}{|u_{hx}|} \, dx + C_{\epsilon} \|c\|_{H^2 (S^{1})}^{2} h^2, \displaybreak[0] \\
|K_6| &\leq \epsilon \int_{S^1} \frac{|(c - c_h)_x|^2}{|u_{hx}|} \, dx + C_{\epsilon} \int_{S^1} (|u_x| - |u_{hx}|)^2 \, dx, \displaybreak[0] \\
|K_7| &\leq C\|c\|_{H^2 (S^{1})}^{2} h^2 + C \int_{S^1} (|u_x| - |u_{hx}|)^2 \, dx.
\end{align*}
Summarizing all these estimates we obtain from \eqref{eq:toest_c} that 
we arrive at
\begin{align*}
\frac{d}{dt} \Big{(} \int_{S^1} & \frac{1}{2} |c - c_{h}|^2 |u_{hx}| \, dx \Big{)} + \int_{S^1} \frac{|c_x - c_{hx}|^2}{|u_{hx}|} \, dx \nonumber \\
&\leq \epsilon \int_{S^1} \frac{|c_x - c_{hx}|^2}{|u_{hx}|} \, dx \nonumber \displaybreak[0] \\
&\quad + \frac{d}{dt} \Big{(} \int_{S^1} (c - c_h) |u_{hx}| (c - I_h(c)) \, dx + \int_{S^1} c (|u_x| - |u_{hx}|) (c - I_h(c)) \, dx \Big{)} \displaybreak[0] \nonumber \\
&\quad + C \int_{S^1} |c - c_h|^2 |u_{hx}| \, dx + C \int_{S^{1}}|u_{t}-u_{ht}|^{2 } dx \nonumber \\
&\quad + C \int_{S^1} |\tau-\tau_{h}|^{2 } dx + C_{\epsilon} \int_{S^1} \big{(} |u_x| - |u_{hx}| \big{)}^2 \, dx + C_{\epsilon} h^2.
\end{align*}
Integrating with respect to time from $0$ to ${\intertime}$, using \eqref{eq:estK3}, \eqref{assu}, and embedding theory we get for $\epsilon$ small enough that
\begin{align*}
\int_{S^1} |c({\intertime}) - & c_{h}({\intertime})|^2 \, dx + \int_0^{\intertime} \int_{S^1} |c_x - c_{hx}|^2 \, dx dt \nonumber \\
&\leq C \int_{S^1} |c_0 - c_{h0}|^2 \, dx + \int_{S^1} |(c_{0} |u_{0x}| - c_{0h} |u_{h0x}|) (c_{0} - I_h(c_{0})) |\, dx \displaybreak[0] \nonumber \\
&\quad + C \hat{\eps} \int_{S^1} |c({\intertime}) - c_h({\intertime})|^2 \, dx \displaybreak[0] + C \int_{S^1} \big{(} |u_x({\intertime})| - |u_{hx}({\intertime})| \big{)}^2 \, dx \displaybreak[0] \nonumber \\
&\quad + C \int_{0}^{{\intertime}} \int_{S^1} |c - c_h|^2 |u_{hx}| \, dx dt + C \int_{0}^{{\intertime}} \int_{S^{1}}|u_{t}-u_{ht}|^{2 } dx dt \displaybreak[0] \nonumber \\
&\quad + C \int_{0}^{{\intertime}} \int_{S^{1}} |\tau-\tau_{h}|^{2 } dx dt + C \int_0^{{\intertime}} \int_{S^1} \big{(} |u_x| - |u_{hx}| \big{)}^2 \, dx dt + C_{\hat{\eps}} h^2.
\end{align*}
Note that
\[
 \int_{S^1} |c_0 - c_{h0}|^2 \, dx = \int_{S^1} |c_0 - I_h(c_0)|^2 \, dx \leq C\|c_0\|_{H^1(S^{1})}^{2} h^2
\]
and, similarly with some arguments as used to estimate $K_3$
\begin{align*}
\int_{S^1} |(c_{0} |u_{0x}| - c_{0h} |u_{0hx}|) (c_{0} - I_h(c_{0})) |\, dx \leq C\|c_0\|_{H^1(S^{1})}^{2} h^2 + C \| u_{0}\|_{H^{2}(S^{1})}^{2} h^{2}.
\end{align*}
Choosing $\hat{\eps}$ small enough and using the above estimates for the initial data yields the claimed estimate \eqref{c-c_h}.
\end{proof}

%%%%% estimate u-u_{h}
\begin{lemma} \label{lem:err_est_C}
Assume that ${\intertime} \in (0,T]$ is such that
\begin{equation} \label{assu0}
\begin{split}
\frac{C^{**}}{2} &\leq |u_{hx}| \leq 2 C^{*} \mbox{ on } [0,{\intertime}], \\
\|c_{h}\|^2_{L^2([0,{\intertime}],H^1(S^1))} &\leq 4 \|c\|^2_{L^2([0,T],H^1(S^1))}.
\end{split}
\end{equation}
Then the following estimate holds for some $C>0$: 
\begin{multline}
\int_0^{\intertime} \int_{S^1} |u_t - u_{ht}|^2 \,dx dt + \int_{S^1} |\tau({\intertime}) - \tau_h({\intertime})|^2 \, dx \\ 
\leq C \int_0^{\intertime} \int_{S^1} |\tau - \tau_h|^2 \, dx dt + C \int_0^{\intertime} \int_{S^1} \big{(} |u_x| - |u_{hx}| \big{)}^2 \, dx dt + C \int_0^{\intertime} \int_{S^1} |c - c_h|^2 \, dx dt + C h^2. \label{eq:est_u}
\end{multline}
\end{lemma}

\begin{proof}
The proof of this lemma follows the lines of the analogous Lemma~5.1 in \cite{D99}. However, some additional terms concerning the dependence on $c$ have to be estimated. More precisely, while the terms $I_1, \dots ,I_5$ as defined below in \eqref{formac} have been treated in \cite{D99} already, the terms $J_1, \dots J_6$ depend on $c$ or $c_{h}$ and are new. They can be dealt with using similar arguments, though.

Let us first write down the difference between the continuous geometric equation \eqref{eq:wpgeo} and its discrete version \eqref{eq:wpgeohB}:
\begin{multline*}
\int_{S^1} (u_t |u_x| - u_{ht} |u_{hx}|) \varphi_h + \Big{(} \frac{u_x}{|u_x|} - \frac{u_{hx}}{|u_{hx}|} \Big{)} \varphi_{hx} \, dx \\
= \int_{S^1} f(c) \nu |u_x| \varphi_h \, dx - \int_{S^{1}} I_{h}(f(c_h)) \nu_{h}|u_{hx}| \varphi_h \, dx \\
- \frac{1}{6} \int_{S^1} (I_{h}(f(c_h)))_{x} \nu_{h} |u_{hx}| h_d^2 \varphi_{hx} \, dx 
+ \frac{1}{6} \int_{S^1} u_{hxt} |u_{hx}| h_d^2 \varphi_{hx} \, dx
\end{multline*}
for all $\varphi_h \in X_h$. As a test function we choose
\begin{align*}
\varphi_h = I_h u_t -u_{ht}= (u_t -u_{ht}) + (I_h u_t -u_t) \in X_h\,.
\end{align*}
Observing that
\[ 
\Big{(} \frac{u_x}{|u_x|} - \frac{u_{hx}}{|u_{hx}|} \Big{)} \cdot (u_{tx}-u_{htx}) = \frac{\partial}{ \partial t} \left( \frac{1}{2}|\tau -\tau_{h}|^{2} |u_{hx}|\right) + u_{xt} \cdot \left ( \tau -\tau_{h} + (\tau_{h} - (\tau \cdot \tau_{h}) \tau) \frac{|u_{hx}|}{|u_{x}|} \right),
\] 
some straightforward calculations show that 
\begin{align}
\int_{S^1} |u_t &- u_{ht}|^2 |u_{hx}| \,dx + \frac{d}{d t} \int_{S^1} \frac{1}{2} |\tau - \tau_h|^2 | u_{hx}| \, dx \nonumber \\ %+ \int_{S^1} \Big{(} \frac{u_x}{|u_x|} - \frac{u_{hx}}{|u_{hx}|} \Big{)} (u_{tx} -u_{htx}) \, dx \\
=& \, \int_{S^1} u_t (|u_{hx}| -|u_x |) (I_h u_t - u_{ht}) \, dx + \frac{1}{6} \int_{S^1} u_{hxt} |u_{hx}| h_d^2 (I_h u_t - u_{ht})_x \, dx \nonumber \\
&+ \int_{S^1} (u_t -u_{ht}) (u_t -I_h u_t) |u_{hx}| \, dx \nonumber \displaybreak[0] \\
&+ \int_{S^1} \Big{(} \frac{u_x}{|u_x|} - \frac{u_{hx}}{|u_{hx}|} \Big{)} (u_t - I_{h} u_t)_x \, dx - \int_{S^1} u_{xt} \cdot \Bigl( \tau -\tau_{h} + (\tau_{h} - (\tau \cdot \tau_{h}) \tau) \frac{|u_{hx}|}{|u_{x}|} \Bigr) \, dx \nonumber \displaybreak[0] \\
&- \int_{S^1} f(c) \nu (|u_{hx}| -|u_x |) (I_h u_t - u_{ht}) \, dx - \frac{1}{6} \int_{S^1} \big{(} I_{h} f(c_h) \big{)}_{x} \nu_{h} |u_{hx}| h_d^2 \big{(} I_h u_t - u_{ht} \big{)}_x \, dx \nonumber \displaybreak[0] \\
&+ \int_{S^1} f(c) (\nu - \nu_{h}) (u_t - u_{ht}) |u_{hx}| \, dx - \int_{S^1} f(c) (\nu -\nu_{h}) (u_t - I_h u_t) |u_{hx}| \, dx \nonumber \displaybreak[0] \\
&+ \int_{S^{1}} \nu_{h} (f(c) - I_{h}(f(c_h))) (u_{t} - u_{ht}) |u_{hx}| \, dx \nonumber \displaybreak[0] \\
&+ \int_{S^{1}} \nu_{h} (f(c) - I_{h}(f(c_h))) (I_{h}u_{t} - u_{t}) |u_{hx}| \, dx \nonumber \\
=& \, I_1 + I_2 + I_3 + I_4 + I_5 + J_1 + J_2 + J_3 + J_4 + J_5 + J_6. \label{formac}
\end{align}

An evaluation of the integrals $I_1$, $I_2$ and $I_3$ is given in \cite[Lemma~5.1]{D99}, therefore we can assert
\begin{align*}
|I_1| &\leq \epsilon \int_{S^1}|u_{ht}-u_t|^2 |u_{hx}|\,dx + C_\epsilon \int_{S^1} (|u_{hx}|-|u_x|)^2 \, dx + C h^4,  \\
I_2 & \leq C h^2, \\ % \big{(} h^2 \| u_t \|_{H^2(S^1)}^2 + \| u_t \|_{H^1(S^1)}^2 \big{)}, \\
|I_3| &\leq \epsilon \int_{S^1} |u_t -u_{ht}|^2 |u_{hx}|\, dx + C_\epsilon h^4,
\end{align*}
with $\epsilon>0$ to be chosen later. 
Note that, for $I_{2}$, one uses Young's inequality $ab \leq a^{2}+ b^{2}/4$ and \eqref{(4.2)extra} to obtain
\begin{align*}
I_{2} &=-\frac{1}{6} \int_{S^{1}} |(I_{h} u_{t})_{x}- u_{hxt}|^{2} |u_{hx}| h_d^2 \, dx + 
\frac{1}{6} \int_{S^{1}} (I_{h} u_{t})_{x} |u_{hx}| h_{d}^{2} ((I_{h} u_{t})_{x} - u_{hxt}) \, dx\\
& \leq \frac{1}{24} \int_{S^{1}} |(I_{h} u_{t})_{x}| |u_{hx}| h_{d}^{2} \, dx \leq Ch^{2}.
\end{align*}
Next we use interpolation \eqref{(4.2)bis} and \eqref{assu0} to obtain that
\[ 
|I_4| = |\int_{S^1} (\tau-\tau_h)\cdot (u_t -I_{h} u_t)_x \, dx | \leq \int_{S^1} |\tau-\tau_h|^2 |u_{hx}| \, dx + C h^2.
\] 
Noting that 
\[
\left |\tau -\tau_{h} + (\tau_{h} -(\tau_{h} \cdot \tau) \tau) \frac{|u_{hx}|}{|u_{x}|} \right| \leq C |\tau -\tau_{h}|  \left |\frac{ |u_x| - |u_{hx}|}{|u_x | } \right| + C |\tau -\tau_{h}|^{2},
\]
by \eqref{assu0} we can infer that
\[ 
|I_5| \leq C \int_{S^1} |\tau_h -\tau |^2 |u_{hx}| \, dx + C \int_{S^1} \big{(} |u_x| -|u_{hx}| \big{)}^2 \, dx.
\] 
The integral $J_1$ can be estimated exactly as $I_1$ because of its similar structure. Using \eqref{assu0}
\[ 
|J_1| \leq \epsilon \int_{S^1} |u_{ht}-u_t|^2 |u_{hx}|\,dx + C_\epsilon \int_{S^1} (|u_{hx}|-|u_x|)^2\, dx + C h^4.
\] 
with $\epsilon>0$ to be chosen later. 
For $J_2$ we note the following using \eqref{assu0}, \eqref{(4.2)extra}, \eqref{IE-1}, \eqref{(4.2)}, and the boundedness of $f'$: 
\begin{align*}
|J_2| &\leq C h^2 \Big{(} \int_{S^1} |(I_{h}(f(c_h)))_{x}|^2 |u_{hx}| \, dx \Big{)}^{1/2} \Big{(} \int_{S^1} |(I_h u_t - u_{ht})_x |^2 |u_{hx}| \, dx \Big{)}^{1/2} \\
&\leq C h \Big{(} \int_{S^1} |(I_{h}(f(c_h)))_{x}|^2 \, dx \Big{)}^{1/2} \Big{(} \int_{S^1} |I_h u_t - u_{ht} |^2 |u_{hx}| \, dx \Big{)}^{1/2} \\
&\leq C h \Big{(} \int_{S^1} |(f(c_h))_{x}|^2 \, dx \Big{)}^{1/2} \Big{(} \int_{S^1} |I_h u_t - u_t |^2 |u_{hx}| \, dx \Big{)}^{1/2} \\
&\quad + C h \Big{(} \int_{S^1} |(f(c_h))_{x}|^2 \, dx \Big{)}^{1/2} \Big{(} \int_{S^1} |u_t - u_{ht} |^2 |u_{hx}| \, dx \Big{)}^{1/2} \displaybreak[0]\\
&\leq C_{\epsilon} h^2 \int_{S^1} |c_{hx}|^2 \, dx + C \| I_h u_t - u_t \|_{L^2(S^1)}^2 + \epsilon \int_{S^1} |u_t - u_{ht} |^2 |u_{hx}| \, dx \displaybreak[0]\\
&\leq \epsilon \int_{S^1} |u_t - u_{ht} |^2 |u_{hx}| \, dx + C_{\epsilon} \Big{(} \int_{S^1} |c_{hx}|^2 \, dx \Big{)} h^2 + Ch^{2}.
\end{align*}
Using Young's inequality, noting that $|\nu - \nu_h| = |\tau - \tau_h|$, and using interpolation estimates we also infer that
\begin{align*}
|J_3| &\leq \epsilon \int_{S^{1}} |u_{t} -u_{ht}|^{2} |u_{hx}| dx + C_\epsilon \int_{S^{1}} |\tau-\tau_{h}|^{2} |u_{hx}| dx,\\
|J_4| &\leq C \int_{S^{1}} |\tau-\tau_{h}|^{2} |u_{hx}| dx + C h^4.
\end{align*}
The second last term $J_5$ can be estimated using \eqref{assu0}, \eqref{(4.2)extra} and the $L^\infty$-bounds for $f$ and $f'$ as follows:
\begin{align*}
|J_5| \leq & \, \epsilon \int_{S^{1}} |u_{t}-u_{ht}|^{2} |u_{hx}| dx + C_{\epsilon} \int_{S^{1}} |f(c)-I_h(f(c_h))|^{2} |u_{hx}| dx \\
\leq & \, \epsilon \int_{S^{1}} |u_{t}-u_{ht}|^{2} |u_{hx}| dx \\
& + C_{\epsilon} \int_{S^{1}} |f(c)-f(c_h)|^{2} dx + C_{\epsilon} \int_{S^{1}} |f(c_h)-I_{h}(f(c_h))|^{2} dx \\
\leq & \, \epsilon \int_{S^{1}} |u_{t} -u_{ht}|^{2} |u_{hx}| dx + C_{\epsilon} \int_{S^{1}} |c-c_h|^{2} dx + C_{\epsilon} h^2 \Big{(}1+ \int_{S^1} |c_{hx}|^2 \, dx \Big{)}.
\end{align*} 
Similarly,
\[ 
|J_6| \leq C \int_{S^{1}} |c-c_h|^{2} dx + C h^2 \Big{(}1+ \int_{S^1} |c_{hx}|^2 \, dx \Big{)}.
\] 
Collecting all the estimates and by embedding theory we obtain from \eqref{formac} that (for $h \leq 1$)
\begin{align*}
\int_{S^1} |u_t &- u_{ht}|^2 |u_{hx}| \,dx + \frac{d}{d t} \int_{S^1} \frac{1}{2} |\tau - \tau_h|^2 | u_{hx}| \, dx \nonumber \\
&\leq C \eps \int_{S^1} |u_t - u_{ht}|^2 |u_{hx}| \, dx + C_{\epsilon} \Big{(} 1 + \int_{S^1} |c_{hx}|^2 \, dx \Big{)} h^2 \nonumber \displaybreak[0] \\
&\quad + C \int_{S^1} |\tau - \tau_h|^2 |u_{hx}| \, dx + C_{\epsilon} \int_{S^1} \big{(} |u_x| - |u_{hx}| \big{)}^2 |u_{hx}| \, dx + C_{\epsilon} \int_{S^1} |c - c_h|^2 \, dx.
\end{align*}
Choosing $\epsilon$ small enough, integrating with respect to time from $0$ to ${\intertime}$ and using \eqref{assu0} we obtain that
\begin{align}
\int_0^{\intertime} \int_{S^1} &|u_t - u_{ht}|^2 \,dx dt + \int_{S^1} |\tau({\intertime}) - \tau_h({\intertime})|^2 \, dx \nonumber \\ 
&\leq C\int_{S^1} | \tau_0 - \tau_{0h} |^2 | u_{0hx} | \, dx + C h^2 \nonumber \displaybreak[0] \\
&\quad + C \int_0^{\intertime} \int_{S^1} |\tau - \tau_h|^2 \, dx dt + C \int_0^{\intertime} \int_{S^1} \big{(} |u_x| - |u_{hx}| \big{)}^2 \, dx dt + C \int_0^{\intertime} \int_{S^1} |c - c_h|^2 \, dx dt. \label{eq:est_u_2}
\end{align}
Note that by
\begin{align}\label{useful} 
|u_x- u_{hx}|^2 = |u_x| |u_{hx}| |\tau -\tau_h |^2 + (|u_x| - |u_{hx}|)^2
\end{align} 
and by interpolation theory \eqref{(4.2)bis} we have that
\[
\int_{S^1} | \tau_0 - \tau_{0h} |^2 | u_{0hx} | \, dx \leq \int_{S^1} \frac{|u_{0x}-u_{0hx}|^2}{|u_{0x}|} \, dx \leq C \int_{S^1} |(u_0 -I_h u_0)_x|^2 \, dx \leq C h^2.
\]
Thus, \eqref{eq:est_u_2} yields the claimed estimate.
\end{proof}

%%%%% second estimate length element
\begin{lemma} \label{lem:err_est_D} 
Assuming that
\begin{equation} \label{assq}
\frac{C^{**}}{2} \leq |u_{hx}| \leq 2 C^{*} \mbox{ on } [0,{\intertime}],
\end{equation}
there exists a constant $C>0$ such that for all $t \in [0,{\intertime}]$
\begin{multline} \label{q-qh}
\int_{S^{1}} (|u_x(t)|-|u_{hx}(t)|)^2 dx \leq C h^2 \\
+ C\int_0^t \int_{S^1} |c-c_h|^2 dx dt' + C \int_0^t \int_{S^{1}} (\tau-\tau_h)^2 \, dx dt' + C \int_0^t \int_{S^{1}} |u_t -u_{ht}|^2 \, dx dt'.
\end{multline}
\end{lemma}

\begin{proof}
Note that thanks to the assumption that $2C^* \geq q_h \geq C^{**} /2$ the discrete length elements are comparable, that is $C^{-1} q_{j+1} \leq q_j \leq c q_{j+1}\,$. 

Integrating \eqref{eq:splitB} with respect to $t$ we obtain
\begin{equation} \label{eq:est_le_first}
(h_jq-q_j)(t)= (h_jq-q_j)(0) -\int_0^t B^+ dt' - \int_0^t B^- dt' -\int_0^t \hat{B} \,dt'\,.
\end{equation}
Clearly
\begin{align*}
|h_jq-q_j|(0)= |h_j |u_{0x}|- h_j |u_{h0x}|_{|_{S_j}}| \leq ch |(u_{0}-I_h u_0)_x|\leq ch\sqrt{h_j}\|u_0 \|_{H^2(S_j)}\,.
\end{align*}
Using \eqref{eq:Bh1}, \eqref{eq:Bh2}, and \eqref{Cstar}
we get (in $S_{j}$)
\begin{align*}
& \int_0^t |\hat{B}|\,dt' \\
&\leq ch_{j} \left(\int_{0}^{t}|u_{t} -r|^{2} dt' \right)^{1/2} \left(\int_{0}^{t}|f(c)-f(c_{j})|^{2} + |f(c)-f(c_{j-1})|^{2}dt' \right)^{1/2} \\
& \quad + C \left( \int_{0}^{t} |f(c_{j})|^{2} + |f(c_{j-1})|^{2} dt' \right)^{1/2} \Bigl( \int_0^t |u_t-r|^2 \bigl( |qh_j - q_{j-1}|^2 +|qh_j - q_{j+1}|^2\bigr)\, dt' \Bigr)^{1/2}\\
& \quad + C \left(\int_{0}^{t}|u_{t} -r|^{2} (q_{j-1} +q_{j+1}) dt' \right)^{1/2} \left (\int_{0}^{t} (|f(c_{j})|^{2} + |f(c_{j-1})|^{2}) (q_{j-1} +q_{j+1}) |\tau- \tau_{j}|^{2} dt' \right)^{1/2} \\
& \quad + C \Bigl (\int_0^t q_{j+1} |f(c_{j})|^2 dt' \Bigr)^{1/2}   \Bigl (\int_0^t (q_j + q_{j+1}) |(u_t-r) -(\dot{u}_j-r_j)|^2 dt' \Bigr)^{1/2}\\
& \quad + C \Bigl (\int_0^t q_{j-1} |f(c_{j-1})|^2 dt' \Bigr)^{1/2} \Bigl (\int_0^t (q_j + q_{j-1}) |(u_t-r) -(\dot{u}_{j-1}-r_{j-1})|^2 \, dt'\Bigr)^{1/2}\,.
\end{align*}
Using \eqref{eq:est_u_t}, the fact that $q_k \leq 2 C^* h_k$ for all $k$, $|f(c) -f(c_{k})| \leq \|f'\|_{L^{\infty}} |c-c_k|$ and the boundedness of $f$, $f'$, $u_{t}$ and $r$ we obtain that
\begin{align*}
\int_0^t |\hat{B}|\,dt' &\leq C h \Bigl ( \int_0^t |c-c_j|^2 +|c-c_{j-1}|^2 dt' \Bigr )^{1/2} \\
& \quad + C \Bigl ( \int_0^t  \bigl( |qh_j - q_{j-1}|^2 +|qh_j - q_{j+1}|^2 \bigr) \, dt' \Bigr )^{1/2}
 + C h \Bigl ( \int_{0}^{t}  |\tau- \tau_{j}|^{2} dt' \Bigr )^{1/2} \\
& \quad + C h \Bigl ( \int_0^t |(u_t-r) -(\dot{u}_j-r_j)|^2 dt' \Bigr )^{1/2}
+ C h \Bigl ( \int_0^t |(u_t-r) -(\dot{u}_{j-1}-r_{j-1})|^2 \, dt'\Bigr )^{1/2}.
\end{align*}
Integrating \eqref{eq:Bplus} with respect to $t$ yields
\begin{align*}
%\int_0^t &|B^+_1 +B^+_2+B^+_3 +B^+_4|\, dt' \\
\int_0^t |B^+| dt' & \leq C\Bigl ( \int_0^t |u_t-r|^2 dt'\Bigr)^{1/2}
\Bigl ( \int_0^t |u_t-r|^2 \bigl( |qh_j -q_j|^2 +|qh_j -q_{j+1}|^2\bigr)\, dt' \Bigr)^{1/2}\\
& \quad + C \Bigl (\int_0^t (q_j + q_{j+1}) |\dot{u}_j-r_j|^2 dt' \Bigr)^{1/2} \Bigl (\int_0^t(q_j + q_{j+1}) |(u_t-r) -(\dot{u}_j-r_j)|^2 dt' \Bigr)^{1/2}\\
& \quad + C\Bigl ( \int_0^t (q_j + q_{j+1}) |u_t-r|^2 dt'\Bigr )^{1/2} \Bigl ( \int_0^t (q_j + q_{j+1}) |(u_t-r) -(\dot{u}_j-r_j)|^2 \, dt'\Bigr )^{1/2}\,.
\end{align*}
Thanks to \eqref{eq:est_u_t}, \eqref{eq:Ah1}, the bounds for $u_{t}$, $r$, and the fact that $q_k \leq 2 C^* h_k$ for all $k$ we obtain
\begin{multline*}
\int_0^t |B^+| dt' %\int_0^t |B^+_1 +B^+_2+B^+_3 +B^+_4|\, dt' 
\leq C \Bigl ( \int_0^t  \bigl( |qh_j -q_j|^2 +|qh_j -q_{j+1}|^2\bigr)\, dt' \Bigr)^{1/2}
 + C h \Bigl (\int_0^t |(u_t-r) -(\dot{u}_j-r_j)|^2 dt' \Bigr)^{1/2} .
\end{multline*}
Repeating the same arguments for $B^-$ and putting all estimates together we infer from \eqref{eq:est_le_first} and recalling \eqref{r-rj} that
\begin{align*}
|h_j q(t)-q_j(t) | & \leq Ch\sqrt{h_j}\| u_{0} \|_{H^2(S_j)} + Ch \left( \int_0^t |c-c_j|^2 +|c-c_{j-1}|^2 dt' \right)^{1/2} \\
& \quad + C h \Bigl (\int_0^t |(u_t -\dot{u}_j)|^2 dt' \Bigr )^{1/2} 
 + C h \Bigl (\int_0^t |(u_t -\dot{u}_{j-1})|^2 dt' \Bigr )^{1/2} \displaybreak[0] \\
& \quad + C\Bigl ( \int_0^t  \bigl( |qh_j -q_j|^2 +|qh_j -q_{j+1}|^2 +|qh_j -q_{j-1}|^2\bigr)\, dt' \Bigr)^{1/2} \\
& \quad + C h \left (\int_{0}^{t} |\tau- \tau_{j}|^{2} + |\tau - \tau_{j-1}|^{2} + |\tau - \tau_{j+1}|^{2} dt' \right)^{1/2}.
\end{align*}
Squaring the above expression, integrating with respect to space over $S_{j}$, using and \eqref{bella1}, \eqref{bella2}, and \eqref{bella0} leads to 
\begin{align}
\int_{S_j} |h_j q(t) - q_j(t)|^2 \, dx 
& \leq C h^4 \Big{(} \| u_0 \|_{H^2(S_j)}^2 + \int_0^t \|c\|^2_{H^1(S_j)} + \|u_t\|_{H^1(M_j)}^2 + \|u\|_{H^2(M_j)}^2 dt' \Big{)} \nonumber \\
& \quad + C h^2 \int_0^t \int_{S_j} |c-c_h|^2 dx dt' + C h^2 \int_0^t \int_{M_j} (q-q_h)^2 dx dt' \nonumber \displaybreak[0] \\
& \quad + C h^2 \int_0^t \int_{M_j} (\tau-\tau_h)^2 \, dx dt' + C h^2 \int_0^t \int_{M_j} |u_t -u_{ht}|^2 \, dx dt'\,,\label{piove}
\end{align}
where $M_j := S_j \cup S_{j+1} \cup S_{j-1}$. Summing up over all grid elements and using that 
\begin{align*}
\int_{S_j} (h_j q -q_j) ^2 \, dx = h^2_j \int_{S_j} ( q -q_h) ^2 \, dx \geq C h^2 \int_{S_j} (q -q_h) ^2 \, dx
\end{align*}
a Gronwall argument yields the claimed estimated \eqref{q-qh}. 
\end{proof}

%%%%%%%%%%%%%%%%%%%%%
\section{Proof of the Convergence Theorem \ref{Thm5.3}}
\label{sec:conv}

Thanks to the estimates in the previous section \ref{sec:errestim} we are ready to prove the main result. We follow the lines of \cite[Theorem~5.3]{D99} but need to also derive the estimates for $c - c_h$ and repeat some arguments for the convenience of the reader. 

\begin{proof}
First of all note that from standard ODE theory we have local existence and uniqueness of a discrete solution $(u_h,c_h)$ of \eqref{eq:wpgeoh}, \eqref{eq:wpspdeh}. So let $T^* \in (0,T]$ be the maximal time for which we have that
\begin{equation} \label{assuT}
\begin{split}
\tfrac{C^{**}}{2} &\leq |u_{hx}| \leq 2 C^{*} \mbox{ on } [0,T^*], \\
%\|c_{h}\|^2_{L^\infty([0,T^*],L^2(S^1))} &\leq 4 \|c\|^2_{L^\infty([0,T],L^2(S^1))}, \\
\|c_{h}\|^2_{L^2([0,T^*],H^1(S^1))} &\leq 4 \|c\|^2_{L^2([0,T],H^1(S^1))}, \mbox{ and}\\
\|c_{h}\|_{C([0,T^*],L^\infty(S^{1}))} &\leq 2 \embconst \|c\|_{C([0,T],H^{1}(S^1))} %C(S^{1})\|c\|_{C([0,T],H^{1}(S^1))}.
\end{split}
\end{equation}
where $\embconst$ is a constant for the embedding $H^1(S^1) \hookrightarrow L^\infty(S^1)$. Inserting equation \eqref{q-qh} into \eqref{eq:est_u} (note that \eqref{assu0} and \eqref{assq} are satisfied thanks to \eqref{assuT}) gives for ${\intertime} \in [0,T^*]$ that
\begin{align*}
& \int_0^{\intertime} \int_{S^1} |u_t - u_{ht}|^2 \,dx dt + \int_{S^1} |\tau({\intertime}) - \tau_h({\intertime})|^2 \, dx \\ 
& \quad \leq C \int_0^{\intertime} \int_{S^1} |\tau - \tau_h|^2 \, dx dt + C \int_0^{\intertime} \int_{S^1} |c - c_h|^2 \, dx dt \\
& \quad \quad + C \int_0^{\intertime} \left( \int_0^t \int_{S^1} |\tau-\tau_h|^2 dx dt' + \int_0^t \int_{S^1} |u_t-u_{ht}|^2 dx dt' \right ) \, dt\\
& \quad \quad + C \int_0^{\intertime} \left( \int_0^t \int_{S^1} |c-c_h|^2 dx dt' \right ) \, dt + C \, h^2 \\
& \quad \leq C \int_0^{\intertime} \int_{S^1} |\tau - \tau_h|^2 \, dx dt + C \int_0^{\intertime} \int_{S^1} |c - c_h|^2 \, dx dt \\
& \quad \quad + C \int_0^{\intertime} \int_0^t \int_{S^1} |u_t-u_{ht}|^2 dx dt' dt + C \, h^2 
\end{align*}
where, for the last inequality, we have used the monotonicity of the integrands. For instance,
\[ 
\int_0^{\intertime} \int_0^t \int_{S^1} |c-c_h|^2 dx dt' dt \leq \int_0^{\intertime} \int_0^{\intertime} \int_{S^1} |c-c_h|^2 dx dt' dt \leq C \int_0^{\intertime} \int_{S^1} |c-c_h|^2 dx dt'.
\] 
A Gronwall argument yields that 
\begin{equation} \label{fastda}
\int_0^{\intertime} \int_{S^1} |u_t - u_{ht}|^2 \,dx dt + \int_{S^1} |\tau({\intertime}) - \tau_h({\intertime})|^2 \, dx \leq C h^2 + C\int_0^{\intertime} \int_{S^1} |c-c_h|^2 dx dt. 
\end{equation}
Inserting this estimate into \eqref{q-qh} gives for $t \in [0, T^*]$ (again using the monotonicity of the integrands)
\begin{equation} \label{q-qh2}
\int_{S^{1}} (|u_x(t)|-|u_{hx}(t)|)^2 dx \leq C h^2 + C \int_0^t \int_{S^1} |c-c_h|^2 dx dt' .
\end{equation}

Next, we plug \eqref{fastda} and \eqref{q-qh2} into \eqref{c-c_h} (note that \eqref{assu} is satisfied thanks to \eqref{assuT}) to obtain for ${\intertime} \in [0,T^*]$ that
\[ 
\int_{S^1} |c({\intertime}) - c_{h}({\intertime})|^2 \, dx + C \int_0^{\intertime} \int_{S^1} |c_x - c_{hx}|^2 \, dx dt \leq  C \int_0^{\intertime} \int_{S^1} |c-c_h|^2 dx dt + C h^{2}.
\] 
Applying Gronwall again yields 
\[ 
\int_{S^1} |c({\intertime}) - c_{h}({\intertime})|^2 \, dx + C \int_0^{\intertime} \int_{S^1} |c_x - c_{hx}|^2 \, dx dt \leq C h^2.
\] 
Inserting this into \eqref{fastda} and \eqref{q-qh2} we obtain that 
\begin{multline} \label{mis}
\int_{S^1} \big{(} |\tau({\intertime}) - \tau_h({\intertime})|^2 + |c({\intertime}) - c_{h}({\intertime})|^2 \big{)} dx + \int_{S^{1}} (|u_x({\intertime})|-|u_{hx}({\intertime})|)^2 dx \\
+ \int_0^{\intertime} \int_{S^1} \big{(} |u_t - u_{ht}|^2 + |c_x - c_{hx}|^2 \big{)} dx dt \leq C \, h^2.
\end{multline}

The constants appearing so far do not depend on $T^*$. Since $u_{hx}$ is constant on each grid interval, the above estimate together with classical embedding theory (see for example \cite[Theorem~2.2]{BGH}) implies
\begin{align*}
|u_{hx}(x,t)| & \geq |u_x (x,t)| - \| |u_x (\cdot, t)|- |u_{hx} (\cdot, t)| \|_{L^\infty (S^1)}\\
& \geq C^{**} - \frac{C}{\sqrt{h}}\| |u_x (\cdot, t)|- |u_{hx} (\cdot, t)| \|_{L^2(S^1)} -c\sqrt{h} \|u(\cdot, t )\|_{H^2(S^1)}\\
& \geq C^{**} - C \sqrt{h} - C \sqrt{h} (\|u_0 \|_{H^2 (S^1)} + \| u_t \|_{L^2 ([0,T], H^2(S^1))})\\
& \geq C^{**} - C \sqrt{h} \geq \frac{3}{4}C^{**}\,,
\end{align*}
for all $h \leq h_0$ with $h_0 \in (0,1)$ sufficiently small independently of $T^*$. Similarly, after eventually decreasing $h_0$ (recall also \eqref{Cstar}), $|u_{hx}| \leq \frac{3}{2} C^*$ for all $h \leq h_0$ independently of $T^*$.

Next observe that using \eqref{mis}, \eqref{IE-2}, and embedding theory we can write for $t \in [0,T^*]$
\begin{align*}
\|c_{h}(t) \|_{L^{\infty}(S^1)} &\leq \|I_{h} c(t) \|_{L^{\infty}(S^{1})} + \| (c_{h} - I_{h} c)(t) \|_{L^{\infty}(S^{1})} \\
& \leq \| c \|_{C([0,T], L^{\infty }(S^{1}))} +\frac{C}{\sqrt{h}} \| (c_{h} - I_{h} c )(t)\|_{L^{2}(S^{1})}\\
& \leq \| c \|_{C([0,T], L^{\infty }(S^{1}))} +\frac{C}{\sqrt{h}} ( \| (c_{h} - c)(t) \|_{L^{2}(S^{1})} + \| (c - I_{h} c)(t) \|_{L^{2}(S^{1})})\\
& \leq \embconst \| c \|_{C([0,T], H^{1}(S^{1}))} +\frac{C}{\sqrt{h}} (h + h\| c \|_{C([0,T], H^{1}(S^{1}))}) \\
& \leq \frac{3}{2} \embconst \| c \|_{C([0,T], H^{1}(S^{1}))}
\end{align*}
for all $h \leq h_0$ independently of $T^*$ (after decreasing $h_0$ if required). Using \eqref{mis} we can easily derive that
\[
\|c_{h}\|^2_{L^2([0,T^*],H^1(S^1))} \leq 3 \|c\|^2_{L^2([0,T],H^1(S^1))}
\]
for all $h \leq h_0$ independently of $T^*$ (after decreasing $h_0$ again if required). Continuity of the solution $(u_h,c_h)$ with respect to time yields a contradiction to the maximality of $T^*$. It follows that $T^*=T$ and the theorem is proved.
\end{proof}

\begin{cor}
\label{coroteo}
Under the assumptions of Theorem~\ref{Thm5.3} we have that
\begin{equation}
\sup_{t \in [0,T]} \| u(t)-u_h(t) \|_{H^1(S^1)}^2 \leq C h^2\,.
\end{equation}
\end{cor}
\begin{proof}
From Theorem~\ref{Thm5.3}, \eqref{useful}, and the fact that $|u_x|$ and $|u_{hx}|$ are bounded we obtain immediately the bound for the semi-norm $|u-u_h|_{H^1 (S^1)}$. To prove the $L^2$-bound note that $u(x,t)-u_h (x,t)= u(x,0)-u_h (x,0) + \int_0^t u_t (x,t')- u_{ht} (x,t') dt'$ and use Theorem~\ref{Thm5.3} again with the interpolation result \eqref{(4.2)} for the initial values.
\end{proof}

%%%%%%%%%%%%%%%%%%%%%
\section{Numerical simulations}
\label{sec:num}

%%%%%
\subsection{Sources and reaction terms}

We now aim for assessing the results in Theorem \ref{Thm5.3}. Exact solutions to the PDE system \mathref{eq:spgeoA}, \mathref{eq:spspdeA} are difficult to obtain whence we prescribe functions $(u,c)$ and account for source terms to ensure that they are solutions, i.e., we consider
\begin{align}
 u_t - \frac{1}{|u_{x}|} \Big{(} \frac{u_{x}}{|u_{x}|} \Big{)}_{x} - f(c) \frac{u_{x}^\perp}{|u_{x}|} =& \, s_u \label{eq:spgeo_s} \\
 c_t + c \frac{|u_x|_t}{|u_x|} - \frac{1}{|u_x|} \Big{(} \frac{c_x}{|u_x|} \Big{)}_x =& \, s_c \label{eq:spspde_s}
\end{align}
with functions $s_u : S^1 \times [0,T] \to \mbR^2$ and $s_c: S^1 \times [0,T] \to \mbR$.

The required extension of the weak formulation \mathref{eq:wpgeo}, \mathref{eq:wpspde} is straightforward. With respect to the spatial discretization of the source terms we apply the interpolation $I_h$ as follows: Instead of the equations \mathref{eq:wpgeoh}, \mathref{eq:wpspdeh} we have
\begin{align}
\int_{S^1} I_h\big{(} u_{ht} \cdot \varphi_{h} \big{)} |u_{hx}| + \frac{u_{hx}}{|u_{hx}|} \cdot \varphi_{hx} \, dx \, &= \int_{S^1} I_h \big{(} f(c_{h})  \varphi_{h} \big{)}\cdot u_{hx}^\perp + I_h \big{(} s_u \cdot \phi_h \big{)} |u_{hx}| \, dx, \label{eq:wpgeo_sh} \\
\frac{d}{dt} \Big{(} \int_{S^1} c_{h} \zeta_{h} |u_{hx}| \, dx \Big{)} + \int_{S^1} \frac{c_{hx} \zeta_{hx}}{|u_{hx}|} \, dx &= \int_{S^1} I_h(s_c) \zeta_h |u_{hx}| \, dx. \label{eq:wpspde_sh}
\end{align}

%%%%%
\subsection{Time discretization}

We apply a semi-implicit scheme which reads as follows:

\begin{prob}[Fully Discrete Scheme] \label{prob:fullydis}
Given a time step $\delta > 0$, let $M = T/ \delta$ and find functions $u_{\delta h}^{(m)}(\cdot) \in X_h$ and $c_{\delta h}^{(m)}(\cdot) \in Y_h$, $m \in \{ 0, \dots, M \}$, of the form
\begin{equation*}
u_{\delta h}^{(m)}(x)= \sum_{j=1}^{N} u_j^{(m)} \varphi_j(x), \quad c_{\delta h}^{(m)}(x) = \sum_{j=1}^{N} c_j^{(m)} \varphi_j(x)
\end{equation*} 
with $u_j^{(m)} \in \R^2$ and $c_j^{(m)} \in \mbR$, such that $u_{\delta h}^{(0)}(\cdot) = u_{h0}$, $c_{\delta h}^{(0)}(\cdot) = c_{h0}$, and such that for all $\varphi_{h} \in X_h$, $\zeta_{h} \in Y_h$ and all $m \in \{ 0, \dots, M-1 \}$

\begin{align}
& \int_{S^1} I_h \Big{(} \frac{u_{\delta h}^{(m+1)} - u_{\delta h}^{(m)}}{\delta} \cdot \varphi_{h} \Big{)} |u_{\delta hx}^{(m)}| + \frac{u_{\delta hx}^{(m+1)}}{|u_{\delta hx}^{(m)}|} \cdot \varphi_{hx} \, dx \nonumber \\
& \quad \quad = \int_{S^1} I_h \big{(} f(c_{\delta h}^{(m)}) \varphi_{h} \big{)} \cdot (u_{\delta hx}^{(m)})^\perp + I_h \big{(} s_u^{(m+1)} \cdot \varphi_{h} \big{)} |u_{\delta hx}^{(m)}| \, dx, \label{eq:wpgeotauh} \\
& \frac{1}{\delta} \Big{(} \int_{S^1} \big{(} c_{\delta h}^{(m+1)} |u_{\delta hx}^{(m+1)}| - c_{\delta h}^{(m)} |u_{\delta hx}^{(m)}| \big{)} \zeta_{h} \, dx \Big{)} + \int_{S^1} \frac{c_{\delta hx}^{(m+1)} \zeta_{hx}}{|u_{\delta hx}^{(m+1)}|} \, dx \nonumber \\
& \quad \quad = \int_{S^1} I_h \big{(} s_c^{(m+1)} \big{)} \zeta_h |u_{\delta hx}^{(m)}|\, dx. \label{eq:wpspdetauh}
\end{align}
\end{prob}

For a more sophisticated time discretization of PDEs on evolving surfaces we refer, for instance, to \cite{LubManVen2013}.
We solve the above fully discrete problem using the following algorithm:

\begin{algo} \label{algo:solver}
Given data: $N$ (number of nodes), $\delta$ (time step), $(u_0, c_0)$ (initial data), $M$ (number of time steps), $tol$ (abort if any segment length becomes smaller).
\begin{enumerate} 
 \item Set $m=0$. \\
 Initialize $u_{h}^{(0)} = u_{h0} = I_h u_{0}$ and $c_{h}^{(0)} = c_{h0} = I_h c_{0}$ by computing the values $u_{i}^{(0)} = u_{0}(x_{i})$ and $c_{i}^{(0)} = c_{0}(x_{i})$, $i = 1, \dots, N$. \\
 Also, compute $q_{i}^{(0)} = | u_{i}^{(0)} - u_{i-1}^{(0)} |$, $i = 1, \dots, N$. \\
 Abort if $\min_j q_j^{(0)} < tol$.
 \item Compute the vertex positions at time $t^{(m+1)} = (m+1) \delta$ from 
\begin{align*}
&\tfrac{1}{2\delta} (q_{i+1}^{(m)} + q_{i}^{(m)}) u_{i}^{(m+1)} - \tfrac{1}{q_{i}^{(m)}} u_{i-1}^{(m+1)} + \Big{(} \tfrac{1}{q_{i+1}^{(m)}} + \tfrac{1}{q_{i}^{(m)}} \Big{)} u_{i}^{(m+1)} - \tfrac{1}{q_{i+1}^{(m)}} u_{i+1}^{(m+1)} \nonumber \\
&\quad = \tfrac{1}{2} (q_{i+1}^{(m)} + q_{i}^{(m)}) \big{(} \tfrac{1}{\delta} u_{i}^{(m)} + s_{u}(x_{i})^{(m+1)} \big{)} + \tfrac{1}{2} f(c_{i}^{(m)}) \big{(} u_{i+1}^{(m)} - u_{i-1}^{(m)} \big{)}^\perp, \quad i=1, \dots, N,
\end{align*}
 and compute $q_{i}^{(m+1)} = | u_{i}^{(m+1)} - u_{i-1}^{(m+1)} |$, $i = 1, \dots, N$. 
 \item Compute the surface field values at time $t^{(m+1)} = (m+1) \delta$ from 
\begin{align*}
&\Big{(} \tfrac{1}{3\delta} (q_{j+1}^{(m+1)} + q_{j}^{(m+1)}) + \big{(} \tfrac{1}{q_{j+1}^{(m+1)}} + \tfrac{1}{q_{j}^{(m+1)}} \big{)} \Big{)} c_{j}^{(m+1)}  \\
&\quad \Big{(} \tfrac{1}{6\delta} q_{j+1}^{(m+1)} - \tfrac{1}{q_{j+1}^{(m+1)}} \Big{)} c_{j+1}^{(m+1)} + \Big{(} \tfrac{1}{6\delta} q_{j}^{(m+1)} - \tfrac{1}{q_{j}^{(m+1)}} \Big{)} c_{j-1}^{(m+1)} \\
&\quad \quad \quad = \tfrac{1}{3\delta} (q_{j+1}^{(m)} + q_{j}^{(m)}) \big{(} c_{j}^{(m)} + \delta s_{c}(x_{j})^{(m+1)} \big{)} \\
&\quad \quad \quad \quad + \tfrac{1}{6\delta} q_{j+1}^{(m)} \big{(} c_{j+1}^{(m)} + \delta s_{c}(x_{j+1})^{(m+1)} \big{)} + \tfrac{1}{6\delta} q_{j}^{(m)} \big{(} c_{j-1}^{(m)} + \delta s_{c}(x_{j-1})^{(m+1)} \big{)}, \quad j = 1, \dots, N.
\end{align*}
 \item If $\min_j q_j^{(m+1)} \geq tol$ and $m+1<M$ then increase $m$ by one and go to step 2.
\end{enumerate}
\end{algo}

Observe that the parametrization doesn't feature any more in the algorithm. The identities in steps two and three are straightforward to compute. For instance, step two is easily be obtained from the continuous version \eqref{eq:odegeo}.

%%%%%
\subsection{A radially symmetric solution}
% compare the radius (computed by length of the polygon) and the average/max/min of c in the nodes with the numerical solution to the ODE system, provide images only.

Consider a radially symmetric setting and denote by $R(t)$ and $B(t)$ the radius of the evolving circle and the constant (in space) value of $c$ along the circle, respectively. We pick $\nu$ to be the outer unit normal of the enclosed ball. Then $v = R'(t)$ and $\kappa = -1/R(t)$. 
The system \mathref{mod1:evol}, \mathref{mod1:spde} becomes
\begin{equation} \label{eq:radsymprob}
 R'(t) = -\frac{1}{R(t)} + f(B(t)), \quad B'(t) + \frac{B(t) R'(t)}{R(t)} = 0.
\end{equation}

We consider the forcing function
\[
 f(B) = 2B-1.
\]
Note that this function is not bounded and thus doesn't satisfy the assumptions of Theorem \ref{Thm5.3}. However, the values of $B$ in the subsequent simulations are bounded. We may therefore think of cutting off $f$ at suitable high and low values which are outside of the computed values and locally smooth it sufficiently. This doesn't alter the computational results but the Theorem then applies.

The constant functions $(R(t),B(t)) = (1,1)$ are a stationary and stable solution to \mathref{eq:radsymprob}. The solution for initial values $R(0) = 1.25$ and $B(t) = 0.8$ converges back to this stable point and has been approximated with a standard MATLAB routine for the comparison in Figure \ref{fig:radsymsol}.

Now let $h = 1/N$ with $N \in \mbN$ and define the initial position of the curve approximation by $$u_j^{(0)} = R(0) (\cos(2\pi j/h), \sin(2\pi j/h))$$ in which we set $c_j^{(0)} = B(0)$. Furthermore, we set $\delta = h^2$. We then perform numerical simulations with the scheme described in Algorithm \ref{algo:solver}. In order to be able to compare with the solution to the ODE system \mathref{eq:radsymprob} we use the length of the computed polygon divided by $2\pi$ and the average of the values of $c_h$ in the nodes,
\[
R_{\delta h}^{(m)} = \frac{1}{2\pi} \sum_{j=1}^N q_j^{(m)}, \quad B_{\delta h}^{(m)} = \frac{1}{N} \sum_{j=1}^N c_j^{(m)}, \qquad m \in \{ 0, \dots, M \}.
\]

Figure \ref{fig:radsymsol} gives a nice impression of the convergence as the computational effort is increased. Note that the errors essentially are due to the spatial discretization. We checked that changing the time step only has a marginal impact on the graphs.

\begin{figure}
\begin{center}
\includegraphics[width=7cm]{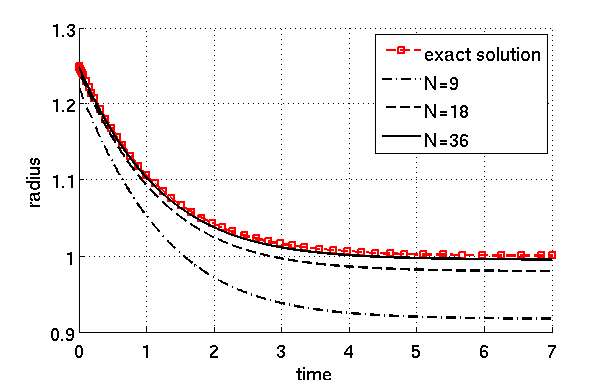} \hfill \includegraphics[width=7cm]{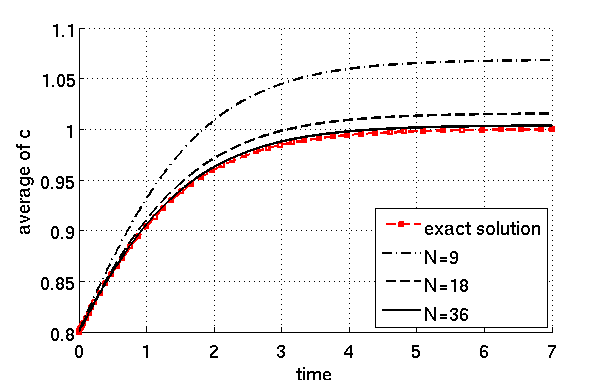}
\end{center}
\caption{Numerical solutions for the radially symmetric solution. The solution $(R(t),B(t))$ to the ODE \mathref{eq:radsymprob} is displayed as well as the solutions $(R_{\delta h},B_{\delta h})$ obtained via Algorithm \ref{algo:solver} for several values of $N$.}
\label{fig:radsymsol}
\end{figure}

%%%%%
\subsection{An oscillating solution}
\label{subsec:sol_oscill}

Consider now the functions
\[
 u(x,t) = 
 \begin{pmatrix} 
  \big{(} 1 + \tfrac{1}{2} \sin(2\pi t) \big{)} \cos(2\pi x) \\
  \big{(} 1 - \tfrac{1}{2} \sin(2\pi t) \big{)} \sin(2\pi x)
 \end{pmatrix}
\]
and
\[
 c(x,t) = t \cos(8\pi x) + (1-t) \sin(6\pi x)
\]
for $x \in [0,1]$ and $t \in [0,T]$ with $T = 1$. Let $f(c) = 2c$ (with regards to the lack of bound the remark in the previous section applies again). Then $(u,c)$ is a solution to \mathref{eq:spgeo_s}, \mathref{eq:spspde_s} if the source terms are given by (writing $s_u = (s_{u1},s_{u2})$)
\begin{align*}
s_{u1} &= \pi \cos(2 \pi  t) \cos(2 \pi  x) \\ 
& - \frac{2 \sqrt{2} \cos(2 \pi  x) (-2+\sin(2 \pi  t)) (t \cos(8 \pi  x) + (1-t) \sin(6 \pi  x))}{\sqrt{9-\cos(4 \pi  t)-4 \sin(2 \pi  (t-2 x))-4 \sin(2 \pi  (t+2 x))}} \\
& + \frac{8 \cos(2 \pi  x) (-2+\sin(2 \pi  t))^2 (2+\sin(2 \pi  t))}{(-9+\cos(4 \pi  t)+4 \sin(2 \pi  (t-2 x))+4 \sin(2 \pi  (t+2 x)))^2}, \displaybreak[0] \\
s_{u2} &= -\pi \cos(2 \pi  t) \sin(2 \pi  x) \\
& + \frac{2 \sqrt{2} \sin(2 \pi  x) (2+\sin(2 \pi  t)) (t \cos(8 \pi  x) + (1-t) \sin(6 \pi  x))}{\sqrt{9-\cos(4 \pi  t)-4 \sin(2 \pi  (t-2 x))-4 \sin(2 \pi  (t+2 x))}} \\
& + \frac{8 \sin(2 \pi  x) (-2+\sin(2 \pi  t)) (2+\sin(2 \pi  t))^2}{(9-\cos(4 \pi  t)-4 \sin(2 \pi  (t-2 x))-4 \sin(2 \pi  (t+2 x)))^2}, \displaybreak[0] \\
s_c &= \cos(8 \pi  x)-\sin(6 \pi  x) \\ 
& + \frac{8 ( 16 t \cos(8 \pi  x)+9 (1-t) \sin(6 \pi  x) )}{9-\cos(4 \pi  t)-4 \sin(2 \pi  (t-2 x))-4 \sin(2 \pi  (t+2 x))} \\ 
& - \frac{128 \cos(2 \pi  x) \sin(2 \pi  t) \sin(2 \pi  x) (3 (-1+t) \cos(6 \pi  x)+4 t \sin(8 \pi  x))}{(9-\cos(4 \pi  t)-4 \sin(2 \pi  (t-2 x))-4 \sin(2 \pi  (t+2 x)))^{2}} \\
& + \frac{4 \pi  \cos(2 \pi  t) (-2 \cos(4 \pi  x)+\sin(2 \pi  t)) (t \cos(8 \pi  x)-(-1+t) \sin(6 \pi  x))}{9-\cos(4 \pi  t)-4 \sin(2 \pi  (t-2 x))-4 \sin(2 \pi  (t+2 x))}.
\end{align*}

For the numerical simulations with Algorithm \ref{algo:solver} we monitored the following errors:
\begin{equation}
\label{eq:err_defs}
\begin{split}
& \mmm{E}_1 := \sup_{t^{(m)}} \int_{S^1} | c - c_h |^2 dx, \quad \mmm{E}_2 := \sup_{t^{(m)}} \int_{S^1} | \tau - \tau_h |^2 dx, \quad \mmm{E}_3 := \sup_{t^{(m)}} \int_{S^1} (|u_x| - |u_{hx}|)^2 dx, \\
& \mmm{E}_4 := \sum_m \delta \int_{S^1} \Big{|} u_t^{(m+1)} - \frac{u_{\delta h}^{(m+1)} - u_{\delta h}^{(m)}}{\delta} \Big{|}^2 dx, \quad \mmm{E}_5 := \sum_m \delta \int_{S^1} | c_x^{(m+1)} - c_{\delta hx}^{(m+1)} |^2 dx
\end{split}
\end{equation}
where we used sufficiently accurate quadrature rules on each interval $S_j$ for the spatial integration. 

We first picked several values for $N$ as displayed in Table \ref{tab:conv_h} and time steps of the size $\delta = h^2$ where $h = 1/N$. We checked that by this choice of the time step the spatial discretization error is dominating. The EOCs of $\mmm{E}_2$, $\mmm{E}_3$, and $\mmm{E}_5$ are close to two which is what Theorem \ref{Thm5.3} asserts. The error in $c_x$ is relatively high but this isn't surprising in view of the spatial oscillations of the surface quantity $c$ which are at a higher frequency than those of the position field $u$. In turn, the EOCs of $\mmm{E}_1$ and $\mmm{E}_4$ are close to four and thus better than Theorem \ref{Thm5.3} predicts, a behavior which may be expected for $\mmm{E}_1$. 

We also assessed the discretization error with respect to the time stepping. For the results in Table \ref{tab:conv_h2} we fixed a very fine spatial mesh with $N=2001$ nodes and varied the time step. Note that our semi-implicit time discretization is of consistency order one. In accordance with this the EOCs of all errors are close to two for all fields. The drops of the EOCs of some errors for small time steps (from about $m=5$ in Table \ref{tab:conv_h2}) are due to the spatial discretization error becoming more significant.

\begin{table}
\begin{center}
\begin{tabular}{|r||l|l|l|l|l|l|l|l|l|l|} \hline 
N & $\mmm{E}_1 \times 10$ & $\eoc_1$ & $\mmm{E}_2 \times 10^2$ & $\eoc_2$ & $\mmm{E}_3$ & $\eoc_3$ & $\mmm{E}_4 \times 10^2$ & $\eoc_4$ & $\mmm{E}_5$ & $\eoc_5$ \\ \hline \hline 
  21 &  1.2912950 & -     & 2.08142 & -     &  1.15896 & -     &  8.047392 & -    &  25.488 & -     \\ \hline 
  61 &  0.0228319 &  3.78 & 0.17548 &  2.32 &  0.04038 &  3.15 &  0.133915 & 3.84 &  2.3531 &  2.23 \\ \hline 
 121 &  0.0015801 &  3.90 & 0.04280 &  2.06 &  0.00637 &  2.70 &  0.009039 & 3.94 &  0.5707 &  2.07 \\ \hline 
 241 &  0.0001023 &  3.97 & 0.01066 &  2.02 &  0.00135 &  2.26 &  0.000581 & 3.98 &  0.1420 &  2.02 \\ \hline 
 401 &  0.0000134 &  3.99 & 0.00384 &  2.01 &  0.00047 &  2.08 &  0.000076 & 3.99 &  0.0511 &  2.01 \\ \hline 
 701 &  0.0000014 &  4.00 & 0.00126 &  2.00 &  0.00015 &  2.03 &  0.000008 & 4.00 &  0.0167 &  2.00 \\ \hline 
1101 &  0.0000002 &  4.00 & 0.00051 &  2.00 &  0.00006 &  2.01 &  0.000001 & 4.00 &  0.0068 &  2.00 \\ \hline 
\end{tabular}
\end{center}
\caption{Errors and EOCs for the test problem described in Section \ref{subsec:sol_oscill}. Note that the errors have been rounded but the EOCs have been computed using the complete numbers. $N$ is the number of nodes, $h = 1/N$ is the spatial step size, and $\delta = h^2$ is the step size in time. The error terms are defined in \mathref{eq:err_defs}.}
\label{tab:conv_h}
\end{table}

\begin{table}
\begin{center}
\begin{tabular}{|r||l|l|l|l|l|l|l|l|l|l|} \hline 
$m$ & $\mmm{E}_1 \times 10 $ & $\eoc_1$ & $\mmm{E}_2 \times 10^2$ & $\eoc_2$ & $\mmm{E}_3$ & $\eoc_3$ & $\mmm{E}_4 \times 10$ & $\eoc_4$ & $\mmm{E}_5 / 10$ & $\eoc_5$ \\ \hline \hline 
0 &  1.23630 & -     &  1.60499 & -     &   1.32017 & -     &  1.03219 & -     &  1.08707  & -    \\ \hline 
1 &  0.45345 &  1.45 &  0.42590 &  1.91 &   0.46027 &  1.52 &  0.33786 &  1.61 &  0.38561  & 1.50 \\ \hline 
2 &  0.14034 &  1.69 &  0.11195 &  1.93 &   0.14138 &  1.70 &  0.09858 &  1.78 &  0.11754  & 1.71 \\ \hline 
3 &  0.03979 &  1.82 &  0.02900 &  1.95 &   0.04043 &  1.81 &  0.02689 &  1.88 &  0.03303  & 1.83 \\ \hline 
4 &  0.01069 &  1.90 &  0.00746 &  1.96 &   0.01099 &  1.88 &  0.00705 &  1.93 &  0.00895  & 1.88 \\ \hline 
5 &  0.00278 &  1.94 &  0.00196 &  1.93 &   0.00290 &  1.92 &  0.00181 &  1.96 &  0.00247  & 1.86 \\ \hline 
6 &  0.00071 &  1.97 &  0.00057 &  1.78 &   0.00076 &  1.94 &  0.00046 &  1.98 &  0.00078  & 1.66 \\ \hline 
7 &  0.00018 &  1.97 &  0.00022 &  1.36 &   0.00021 &  1.88 &  0.00012 &  1.98 &  0.00035  & 1.16 \\ \hline 
\end{tabular}
\end{center}
\caption{Errors and EOCs for the test problem described in Section \ref{subsec:sol_oscill}. Note that the errors have been rounded but the EOCs have been computed using the complete numbers. The time step size is given by $\delta = 0.02 \times 2^{-m}$, the spatial step size is fixed at $h = 1/N$ where $N = 2001$ is the number of nodes. The error terms are defined in \mathref{eq:err_defs}.}
\label{tab:conv_h2}
\end{table}

%%%%%%%%%%%%%%%%%%%%%
\bibliography{refs}
\bibliographystyle{acm}

\end{document}